\documentclass[10pt,twoside]{article}
\usepackage[english]{babel}
\usepackage{amsmath}
\usepackage[latin1]{inputenc}
\usepackage{amsfonts}
\usepackage{amssymb}
\usepackage{amsthm}
\usepackage{stmaryrd}
\usepackage{MnSymbol}
\usepackage{microtype}
\usepackage{mathrsfs}
\usepackage[all]{xy}
\usepackage{enumitem}

\theoremstyle{theorem}
\newtheorem{Theorem}{Theorem}[section]
\newtheorem{Proposition}[Theorem]{Proposition}
\newtheorem{Lemma}[Theorem]{Lemma}
\newtheorem{Conjecture}[Theorem]{Conjecture}
\newtheorem{Corollary}[Theorem]{Corollary}

\theoremstyle{definition} 
\newtheorem{Definition}[Theorem]{Definition}
\newtheorem{Definition*}[]{Definition}

\newtheorem{Remark}[Theorem]{Remark}

\setenumerate{label=\emph{(\alph*)}, ref=(\alph*)}
\DeclareMathOperator{\et}{{\textup{\'et}}}
\DeclareMathOperator{\tame}{{tame}}
\DeclareMathOperator{\cont}{{cont}}
\DeclareMathOperator{\ab}{{ab}}
\DeclareMathOperator{\Strat}{Strat}
\DeclareMathOperator{\Repf}{Repf}
\DeclareMathOperator{\rs}{rs}

\DeclareMathOperator{\Pic}{Pic}
\DeclareMathOperator{\Alb}{Alb}
\DeclareMathOperator{\alb}{alb}
\DeclareMathOperator{\Spec}{Spec}
\DeclareMathOperator{\Frac}{Frac}
\DeclareMathOperator{\Vectf}{Vectf}
\DeclareMathOperator{\Exp}{Exp}
\DeclareMathOperator{\rank}{rank}
\DeclareMathOperator{\Cl}{Cl}
\DeclareMathOperator{\coker}{coker}
\DeclareMathOperator{\id}{id}
\DeclareMathOperator{\red}{red}
\DeclareMathOperator{\NS}{NS}
\DeclareMathOperator{\Hom}{Hom}
\DeclareMathOperator{\alg}{alg}
\DeclareMathOperator{\Ext}{Ext}
\DeclareMathOperator{\Top}{top}
\DeclareMathOperator{\Char}{char}

\newcommand{\llparen}{(\!(}
\newcommand{\rrparen}{)\!)}
\newcommand{\Z}{\mathbb{Z}}
\newcommand{\G}{\mathbb{G}}
\newcommand{\F}{\mathbb{F}}

\newcommand{\Q}{\mathbb{Q}}
\newcommand{\N}{\mathbb{N}}
\newcommand{\C}{\mathbb{C}}
\newcommand{\A}{\mathbb{A}}
\renewcommand{\P}{\mathbb{P}}
\newcommand{\plim}{{\textstyle \varprojlim_p}}

\def\YEAR{\year}\newcount\VOL\VOL=\YEAR\advance\VOL by-1995
\def\firstpage{1}\def\lastpage{1000}
\def\received{}\def\revised{}
\def\communicated{}

\makeatletter
\def\magnification{\afterassignment\m@g\count@}
\def\m@g{\mag=\count@\hsize6.5truein\vsize8.9truein\dimen\footins8truein}
\makeatother

\font\eightrm=cmr8
\font\caps=cmcsc10                    
\font\Caps=cmcsc10 scaled \magstep1   

\makeatletter
\@specialpagefalse\headheight=8.5pt
\def\DocMath{}
\renewcommand{\@evenhead}{%
    \ifnum\thepage>\lastpage\rlap{\thepage}\hfill%
    \else\rlap{\thepage}\slshape\leftmark\hfill{\caps\SAuthor}\hfill\fi}%
\renewcommand{\@oddhead}{%
    \ifnum\thepage=\firstpage{\DocMath\hfill\llap{\thepage}}%
    \else{\slshape\rightmark}\hfill{\caps\STitle}\hfill\llap{\thepage}\fi}%
\makeatother

\def\TSkip{\bigskip}
\newbox\TheTitle{\obeylines\gdef\GetTitle #1
\ShortTitle  #2
\SubTitle    #3
\Author      #4
\ShortAuthor #5
\EndTitle
{\setbox\TheTitle=\vbox{\baselineskip=20pt\let\par=\cr\obeylines%
\halign{\centerline{\Caps##}\cr\noalign{\medskip}\cr#1\cr}}%
	\copy\TheTitle\TSkip\TSkip%
\def\next{#2}\ifx\next\empty\gdef\STitle{#1}\else\gdef\STitle{#2}\fi%
\def\next{#3}\ifx\next\empty%
    \else\setbox\TheTitle=\vbox{\baselineskip=20pt\let\par=\cr\obeylines%
    \halign{\centerline{\caps##} #3\cr}}\copy\TheTitle\TSkip\TSkip\fi%
\centerline{\caps #4}\TSkip\TSkip%
\def\next{#5}\ifx\next\empty\gdef\SAuthor{#4}\else\gdef\SAuthor{#5}\fi%
\ifx\received\empty\relax
    \else\centerline{\eightrm Received: \received}\fi%
\ifx\revised\empty\TSkip%
    \else\centerline{\eightrm Revised: \revised}\TSkip\fi%
\ifx\communicated\empty\relax
    \else\centerline{\eightrm Communicated by \communicated}\fi\TSkip\TSkip%
\catcode'015=5}}\def\Title{\obeylines\GetTitle}
\def\Abstract{\begingroup\narrower
    \parskip=\medskipamount\parindent=0pt{\caps Abstract. }}
\def\EndAbstract{\par\endgroup\TSkip}

\long\def\MSC#1\EndMSC{\def\arg{#1}\ifx\arg\empty\relax\else
     {\par\narrower\noindent%
     2010 Mathematics Subject Classification: #1\par}\fi}

\long\def\KEY#1\EndKEY{\def\arg{#1}\ifx\arg\empty\relax\else
	{\par\narrower\noindent Keywords and Phrases: #1\par}\fi\TSkip}

\newbox\TheAdd\def\Addresses{\vfill\copy\TheAdd\vfill
    \ifodd\number\lastpage\vfill\eject\phantom{.}\vfill\eject\fi}
{\obeylines\gdef\GetAddress #1
\Address #2 
\Address #3
\Address #4
\EndAddress
{\def\xs{4.3truecm}\parindent=0pt
\setbox0=\vtop{{\obeylines\hsize=\xs#1\par}}\def\next{#2}
\ifx\next\empty 
     \setbox\TheAdd=\hbox to\hsize{\hfill\copy0\hfill}
\else\setbox1=\vtop{{\obeylines\hsize=\xs#2\par}}\def\next{#3}
\ifx\next\empty 
     \setbox\TheAdd=\hbox to\hsize{\hfill\copy0\hfill\copy1\hfill}
\else\setbox2=\vtop{{\obeylines\hsize=\xs#3\par}}\def\next{#4}
\ifx\next\empty\ 
     \setbox\TheAdd=\vtop{\hbox to\hsize{\hfill\copy0\hfill\copy1\hfill}
                \vskip20pt\hbox to\hsize{\hfill\copy2\hfill}}
\else\setbox3=\vtop{{\obeylines\hsize=\xs#4\par}}
     \setbox\TheAdd=\vtop{\hbox to\hsize{\hfill\copy0\hfill\copy1\hfill}
	        \vskip20pt\hbox to\hsize{\hfill\copy2\hfill\copy3\hfill}}
\fi\fi\fi\catcode'015=5}}\gdef\Address{\obeylines\GetAddress}

\begin{document}
\Title
Evidence for a Generalization of Gieseker's Conjecture on
Stratified Bundles in Positive Characteristic
\ShortTitle 
Evidence for a Generalization of Gieseker's Conjecture
\SubTitle   
\Author 
Lars Kindler
\ShortAuthor 
Lars Kindler
\EndTitle
\Abstract 
Let $X$ be a smooth, connected, projective variety over an algebraically
closed field of positive characteristic. In \cite{Gieseker/FlatBundles},
Gieseker conjectured that every stratified bundle (i.e.~every
$\mathcal{O}_X$-coherent $\mathscr{D}_{X/k}$-module) on $X$ is trivial, if and
only if $\pi_1^{\et}(X)=0$.  This was proven by Esnault-Mehta,
\cite{EsnaultMehta/Gieseker}. 

Building on the classical situation over the complex numbers,  we present and motivate a generalization of Gieseker's conjecture, using the notion of regular singular stratified bundles developed in the author's thesis and \cite{Kindler/FiniteBundles}.  In the main part of this article we establish some important special cases of this generalization; most notably we prove that for not necessarily proper $X$, $\pi_1^{\tame}(X)=0$ implies that there are no nontrivial regular singular stratified bundles with abelian monodromy.
\EndAbstract
\MSC 
14E20, 14E22, 14F10
\EndMSC
\KEY 
Fundamental group, coverings, stratified bundles, D-modules, tame ramification
\EndKEY
\Address 
Lars Kindler
Freie Universit\"at Berlin
Mathematisches Institut
Arnimallee 3
14195 Berlin, Germany
\Address
\Address
\Address
\EndAddress
\section{Introduction and statement of the conjecture}
Let $X$ be a smooth, separated, connected scheme of finite type over an algebraically
closed field $k$, and fix a base point $x\in X(k)$. For readability, we write
$\Pi^{\et}_X:=\pi_1^{\et}(X,x)$, and we denote by 
$\Repf_k^{\cont}\Pi_X^{\et}$ the category of continuous representations of the
profinite group
$\Pi_X^{\et}$ on finite dimensional
$k$-vector spaces equipped with the discrete topology.

If $k=\C$, then $\Pi_X^{\et}$ is the profinite completion of the abstract group
$\Pi_X^{\Top}:=\pi_1^{\Top}(X(\C),x)$, and if we write
$\Repf_{\C}\Pi_X^{\Top}$ for the category of representations of
$\Pi_X^{\Top}$ on finite dimensional $\C$-vector spaces, then $\Repf^{\cont}_{\C}\Pi^{\et}_X$ can be considered as
the full subcategory of $\Repf_{\C} \Pi^{\Top}_{X}$  having as objects
precisely those representations which factor through a finite group. Since
$\Pi_X^{\Top}$ is finitely generated, the category $\Repf_{\C}\Pi^{\Top}_{X}$  is controlled
by $\Repf_{\C}^{\cont}\Pi_X^{\et}$, according to the following theorem:

\begin{Theorem}[Grothendieck \cite{Grothendieck/Malcev}, Malcev \cite{Malcev}]\label{thm:malcev}
	If $\phi:G\rightarrow H$ is a morphism of finitely generated groups,
	then the
	following statements are equivalent:
	\begin{enumerate}
		\item\label{malcev1} The induced morphism $\widehat{G}\rightarrow \widehat{H}$ is an isomorphism, where
	$\widehat{(-)}$ denotes profinite completion.
\item\label{malcev2} The induced
	$\otimes$-functor \[\Repf_{\C} H\rightarrow \Repf_{\C} G\] is a
	$\otimes$-equivalence.
\item\label{malcev3} The induced $\otimes$-functor \[\Repf^{\cont}_{\C} \widehat{H}\rightarrow
	\Repf^{\cont}_{\C} \widehat{G}\] is
	a $\otimes$-equivalence.
	\end{enumerate}
\end{Theorem}

Accordingly, if $f:Y\rightarrow X$ is a morphism of smooth, connected, complex varieties, then the induced morphism $\Pi_Y^{\et}\rightarrow \Pi_X^{\et}$ (with respect
to compatible base points) is an isomorphism if and only if
$\Repf_{\C}\Pi_X^{\Top}\rightarrow \Repf_{\C}\Pi_{Y}^{\Top}$ is a $\otimes$-equivalence.
This consequence was already noted in \cite{Grothendieck/Malcev}.

To study the category $\Repf_{\C}\Pi_X^{\Top}$, we invoke 
the Riemann-Hilbert correspondence as developed in
\cite{Deligne/RegularSingular}: It states that the choice of a base point $x\in X(\C)$ gives a
$\otimes$-equivalence $|_x$ of Tannakian categories between the category of regular singular
flat connections on $X$ and the category $\Repf_{\C}\Pi_X^{\Top}$. 
Theorem \ref{thm:malcev} then translates into the following completely algebraic
statement, in which we suppress the choice of base points from the notation:

\begin{Corollary}\label{cor:giesekerchar0}
	If $f:Y\rightarrow X$ is a morphism of smooth, connected, separated, finite type
	$\C$-schemes, then the following are equivalent:
	\begin{enumerate}
		\item The morphism $\Pi^{\et}_Y\rightarrow \Pi^{\et}_X$ induced by $f$ is
			an isomorphism.
		\item Pull-back along $f$ induces an equivalence on the
			categories of 
			regular singular flat connections.
		\item Pull-back along $f$ induces an equivalence on the
			categories of 
			regular singular flat connections with \emph{finite monodromy}.
	\end{enumerate}
	
	In particular: $\Pi_Y^{\et}=0$ if and only if every regular singular flat
	connection on $Y$ is trivial.
\end{Corollary}

This article is devoted to the study of analogous statements in positive
characteristics. Let
$k$ be an algebraically closed field of characteristic $p>0$, and $X$ a smooth, connected,
separated, finite type $k$-scheme. In general, neither the category of vector bundles with flat
connection, nor the category of coherent $\mathcal{O}_X$-modules with flat connection  are
Tannakian categories over $k$. Instead, we consider
left-$\mathscr{D}_{X/k}$-modules which are coherent (and then automatically locally
free) as $\mathcal{O}_X$-modules, where $\mathscr{D}_{X/k}$ is the ring of differential
operators relative to $k$, as developed in \cite[\S16]{EGA4}. 
Following
\cite{Grothendieck/Crystals} and \cite{Saavedra}, we call such objects \emph{stratified bundles}, and we write $\Strat(X)$ for the
category of stratified bundles on $X$.
Recall that in characteristic $0$, giving a stratified bundle is
 equivalent to giving a vector
bundle with flat connection. In positive characteristic, these notions are not
equivalent.  A stratified bundle is called
trivial if it is isomorphic to $\mathcal{O}_X^{\oplus n}$ with the canonical diagonal
left-$\mathscr{D}_{X/k}$-action. In \cite{Kindler/FiniteBundles} (see also
\cite{Kindler/thesis}), a notion of regular singularity for stratified bundles is
defined and studied, generalizing work of Gieseker
(\cite{Gieseker/FlatBundles}); for a summary see Section \ref{sec:rs}. We write $\Strat^{\rs}(X)$ for the full subcategory of $\Strat(X)$
with objects regular singular stratified bundles; after choosing a base point,
it is a neutral Tannakian category over $k$.

Using the theory of Tannakian categories, there
still is a procedure to attach to a stratified bundle a monodromy group and a monodromy
representation at the base point
$x\in X(k)$. The main result of \cite{Kindler/FiniteBundles} states that this
procedure induces a $\otimes$-equivalence between the category of regular singular
stratified bundles with finite monodromy and $\Repf_{k}^{\cont}\Pi_X^{\tame}$, where
$\Pi_X^{\tame}:=\pi_1^{\tame}(X,x)$ is the tame fundamental group as defined in
\cite{Kerz/tameness}. This result suggests the following conjecture,
completely analogous to
Corollary \ref{cor:giesekerchar0}:

\begin{Conjecture}\label{conj:generalGieseker}
	Let $k$ be an algebraically closed field of characteristic $p>0$.
	If $f:Y\rightarrow X$ is a morphism of smooth, connected, separated, finite type
	$k$-schemes, then the following are equivalent:
	\begin{enumerate}
		\item\label{conj1} The morphism $\Pi^{\tame}_Y\rightarrow \Pi^{\tame}_X$ induced by $f$ is
			an isomorphism.
		\item\label{conj2} Pull-back along $f$ induces an equivalence
			on the categories of regular
			singular stratified bundles.		
		\item\label{conj3} Pull-back along $f$ induces an equivalence
			on the categories of 
			regular singular stratified bundles with \emph{finite monodromy}.
	\end{enumerate}
	
	In particular: $\Pi_Y^{\tame}=0$ if and only if every regular singular stratified
	bundle on $Y$ is trivial.
\end{Conjecture}

For a slightly different exposition of this conjecture, see
\cite{Esnault/ECM}.
Note that by the main result of \cite{Kindler/FiniteBundles} we always have
\ref{conj2}$\Rightarrow$\ref{conj3}$\Leftrightarrow$\ref{conj1}.

The main part of this article is concerned with establishing special cases of and evidence
for the validity of Conjecture \ref{conj:generalGieseker}, i.e.~for the direction
\ref{conj1}$\Rightarrow$\ref{conj2}. We give a brief summary: 	Gieseker's original conjecture from \cite{Gieseker/FlatBundles} corresponds to
		Conjecture \ref{conj:generalGieseker} for a projective
		morphism $f:Y\rightarrow \Spec (k)$. His conjecture
		was proven by Esnault-Mehta, \cite{EsnaultMehta/Gieseker}. Thus Conjecture
		\ref{conj:generalGieseker} generalizes Gieseker's conjecture in two
		``directions'': It allows non-projective varieties by using the
		notion of regular singularity, and it gives a relative
		formulation.

		The main result of this article, proven in Section
		\ref{sec:extensions}, is the following:
		\begin{Theorem}[see Theorem
				\ref{thm:abelianQuotient}]\label{thm:abelianQuotientINTRO}
			In the situation of Conjecture
			\ref{conj:generalGieseker}, assume that $X=\Spec(k)$,
			and that $Y$ admits a good compactification. If
			$\Pi_Y^{\ab,(p')}=0$ and if $\Pi_Y^{\tame}$
			does not have a quotient isomorphic to $\Z/p\Z$, then
			every regular singular stratified bundle with abelian
			monodromy is trivial.

			In other words: The abelianization of the
			pro-algebraic group associated with
			$\Strat^{\rs}(Y)$ (and the choice of a base point) is trivial.
		\end{Theorem}
Here a \emph{good compactification} of $Y$ is a smooth, proper $k$-scheme
$\overline{Y}$ containing $Y$ as a dense open subscheme, such that
$\overline{Y}\setminus Y$ is the support of a strict normal crossings divisor,
and $(-)^{\ab,(p')}$ denotes the maximal abelian pro-prime-to-$p$-quotient.
This quotient is independent of the choice of the base point.

Theorem \ref{thm:abelianQuotientINTRO} is a consequence of the fact that
under the given assumptions there are no nontrivial rank $1$ stratified bundles,
and no nontrivial regular singular extensions of two rank $1$ stratified
bundles. Hence we establish these facts first:
in Section \ref{sec:linebundles} we study stratified bundles of rank $1$, and
we prove a relative version of Conjecture \ref{conj:generalGieseker} for
stratified line bundles: 
\begin{Theorem}[see Theorem \ref{thm:rank1}]\label{thm:rank1INTRO}
	Let $X$ and $Y$ be smooth, connected, separated, finite type $k$-schemes, which
	admit good compactifications $\overline{X}$ and $\overline{Y}$ over $k$. 
	If $f:Y\rightarrow X$ is a map extending to a
	morphism $\bar{f}:\overline{Y}\rightarrow \overline{X}$ such that $f$
	induces an isomorphism
	$\Pi_Y^{\ab,(p')}\xrightarrow{\cong}\Pi_X^{\ab,(p')}$ and such
	that the cokernel of the induced map
	$H^0(X,\mathcal{O}_X^{\times})\rightarrow H^0(Y,\mathcal{O}_Y^\times)$
	is a $p$-group,
	then pull-back along $f$ induces an isomorphism
	\[\Pic^{\Strat}(X)\xrightarrow{\cong} \Pic^{\Strat}(Y)\]
	where $\Pic^{\Strat}$ denotes the group of isomorphism classes of
	stratified line bundles.
\end{Theorem}
Note that contrary to the situation over the complex
numbers, a stratified line bundle in our context is always regular singular 
(Proposition \ref{prop:rsfacts}).
The assumption on the cokernel of
$H^0(X,\mathcal{O}_X^{\times})\rightarrow H^0(Y,\mathcal{O}_{Y}^{\times})$ is
trivially fulfilled if $Y$ and $X$ are proper. For further comments see
Remark \ref{rem:globalunits}. 

In particular we obtain:
\begin{Corollary}
	Let $Y,X$ be proper, smooth, connected $k$-schemes.
	If $f:Y\rightarrow X$ is a morphism such that $f$ induces an isomorphism 
	$\Pi^{\ab,(p')}_Y\xrightarrow{\cong}\Pi_X^{\ab, (p')}$, then $f$ induces an isomorphism
	$\Pic^{\Strat}(X)\xrightarrow{\cong} \Pic^{\Strat}(Y)$.
\end{Corollary}

In the case that $X=\Spec k$, the assumption on the existence of a good
compactification of $Y$ is not necessary:
\begin{Proposition}[see Proposition
		\ref{prop:rank1absolute}]\label{prop:rank1absoluteINTRO}
	Let $Y$ be a smooth, connected, separated, finite type $k$-scheme,
	such that $\Pi_Y^{\ab,(p')}=0$. Then $\Pic^{\Strat}(Y)=0$.
\end{Proposition}

In Section \ref{sec:extensions} we study regular singular extensions of rank
$1$ stratified bundles, and prove Theorem \ref{thm:abelianQuotientINTRO}.

\vspace{1em}
For regular singular stratified bundles with not necessarily abelian monodromy, we establish the following two special cases of
Conjecture \ref{conj:generalGieseker}:

It is well-known that $\pi_1^{\tame}(\A^n_k)=0$ and in Section
\ref{sec:affinespaces} we give a short proof of:
\begin{Theorem}[see Theorem \ref{thm:affinespaces}]Every
	regular singular stratified bundle on $\A^n_k$ is trivial.
\end{Theorem}
This was already sketched in \cite{Esnault/ECM} using a slightly different
approach.\\

If $f:Y\rightarrow X$ is a universal homeomorphism, then $f$ induces an
isomorphism
$\Pi_Y^{\tame}\xrightarrow{\cong}\Pi_X^{\tame}$ (\cite{Vidal/TopologicalInvariance}), and in Section
\ref{sec:univHomeo} we show:
\begin{Theorem}[see Theorem \ref{thm:topologicalInvarianceGeneralCase}]If
	$f:Y\rightarrow X$ is a universal homeomorphism, then pull-back along
	$f$ is an equivalence $\Strat^{\rs}(X)\rightarrow \Strat^{\rs}(X)$.
\end{Theorem}
In Section
\ref{sec:rs} we gather a few general facts about regular singular stratified
bundles which will be needed in the course of the text.

\vspace{2em}

\emph{Acknowledgements:} Large parts of the contents of this note were
developed while the author wrote his doctoral dissertation under the supervision of
H\'el\`ene Esnault. Her influence is undeniable, and the author is
immensely grateful for her support. The author also wishes to thank Oliver Br\"aunling
and Jakob Stix
for helpful discussions, the referee for carefully reading the manuscript, and Brian Conrad for pointing out a helpful result by
Lang; see the discussion after
Proposition \ref{prop:rank1absolute}.

This work was supported by the Sonderforschungsbereich/Transregio 45 ``Periods,
moduli spaces and the arithmetic of algebraic varieties'' of the DFG and
the European Research Council Advanced Grant 0419744101

\section{Generalities on stratified bundles}\label{sec:rs}
Let $k$ be an algebraically closed field of characteristic $p>0$, and fix a
smooth, separated, connected, finite type $k$-scheme $X$.
\begin{Definition}A \emph{stratified bundle} on $X$ is a
	$\mathscr{D}_{X/k}$-module which is $\mathcal{O}_X$-coherent. A
	morphism of stratified bundles is a morphism of
	$\mathscr{D}_{X/k}$-modules. A stratified bundle $E$ is
	called \emph{trivial} if it is isomorphic to $\mathcal{O}_X^n$ with
	its canonical diagonal $\mathscr{D}_{X/k}$-action. Here
	$\mathscr{D}_{X/k}$ is the ring of differential operators on $X$, see
	\cite[\S16]{EGA4}.

	We write $\Strat(X)$ for the category of stratified bundles.
\end{Definition}

\begin{Remark}The name ``stratified module'' goes back to Grothendieck: In
	\cite{Grothendieck/Crystals} he defines the notion of a
	stratification  relative to $k$ on an $\mathcal{O}_X$-module $E$ as an
	``infinitesimal descent datum'', and since $X$ is smooth over $k$,
	such a datum is equivalent to the datum of a
	$\mathscr{D}_{X/k}$-action on $E$, compatible with its
	$\mathcal{O}_X$-structure, see
	e.g.~\cite[Ch.~2]{BerthelotOgus/Crystalline}. As in characteristic $0$, one shows that a
	$\mathcal{O}_X$-coherent $\mathscr{D}_{X/k}$-module is automatically
	locally free, \cite[2.17]{BerthelotOgus/Crystalline}. Hence the name ``stratified bundle''.
\end{Remark}
	
The following result of Katz gives a different perspective on the notion of a
stratified bundle:
\begin{Theorem}[{Katz,
		\cite[Thm.~1.3]{Gieseker/FlatBundles}}]\label{thm:cartier}
	Let $F:X\rightarrow X$ denote the absolute Frobenius.
	The category $\Strat(X)$ is equivalent to the following category:
	\begin{description}
		\item[Objects:] Sequences of pairs
	$E:=(E_n,\sigma_n)_{n\geq 0}$ with $E_n$ a vector
	bundle on $X$ and $\sigma_n:E_n\xrightarrow{\cong} F^*E_{n+1}$ an
	$\mathcal{O}_X$-linear isomorphism. 
\item[Morphisms:]A morphism
	$\phi:(E_n,\sigma_n)\rightarrow (E'_n,\sigma'_n)$
	is a
	sequence of morphisms of vector bundles $\phi_n:E_n\rightarrow E'_n$,
	such that $F^*\sigma_{n+1}=\sigma'_n\phi_n$.
	\end{description}
	The functor giving the equivalence assigns to a stratified bundle $E$
	the sequence $(E_n,\sigma_n)_{n\geq 0}$, with 
	\[E_n(U)=\left\{ e\in E(U)| D(e)=0\text{ for all
	}D\in\mathscr{D}^{\leq p^n}_{X/k}(U)\text{ with
	}D(1)=0  \right\}
	\]
	and $\sigma_n:E_n\rightarrow F^*E_{n+1}$ the isomorphism given by
	Cartier's theorem \cite[\S 5]{Katz/Connections}. This functor is
	compatible with tensor products.
\end{Theorem}

In the sequel we will freely switch between the two perspectives on stratified
bundles.
The description by the above theorem is especially nice when $X$ is proper
over $k$:

\begin{Proposition}[{\cite[Prop.~1.7]{Gieseker/FlatBundles}}]
	If $X$ is proper over $k$, then the isomorphism class of a stratified
	bundle $E=(E_n,\sigma_n)_{n\geq 0}$ only depends on the isomorphism classes of
	the vector bundles $E_1,E_2,\ldots$.
\end{Proposition}

\subsection{Regular singular stratified bundles}
	We recall from \cite{Kindler/FiniteBundles} the definition of regular
	singular stratified bundles. 
	\begin{Definition}
	Let $X$ be a smooth, separated, finite type $k$-scheme.
	\begin{enumerate}[label={(\alph*)}]\label{defn:diffops}
		\item If $\overline{X}$ is a smooth, separated, finite type
			$k$-scheme containing $X$ as as an open dense
			subscheme such that $\overline{X}\setminus X$ is a
			strict normal crossings divisor, then the pair
			$(X,\overline{X})$ is called \emph{good partial
			compactification of $X$}. If $\overline{X}$ is also proper, then
			$(X,\overline{X})$ is called \emph{good
			compactification of $X$.}
		\item\label{item:diffops} If $(X,\overline{X})$ is a good partial
			compactification, then
			$\mathscr{D}_{\overline{X}/k}(\log
			\overline{X}\setminus X)$ is by definition the sheaf
			of subalgebras of $\mathscr{D}_{\overline{X}/k}$
			spanned over an open $\overline{U}\subset
			\overline{X}$ by those differential operators in
			$\mathscr{D}_{\overline{X}/k}(\overline{U})$, fixing
			all powers of the ideal of the divisor
			$(\overline{X}\setminus X)
			\overline{U}$. Lets make this explicit: If
			$x_1,\ldots, x_n\in
			H^0(\overline{U},\mathcal{O}_{\overline{U}})$ are
			coordinates on $\overline{U}$, i.e.~if they define an \'etale morphism
			$\overline{U}\rightarrow \A^n_k$, then
			$\mathscr{D}_{\overline{U}/k}$ is spanned by operators
			$\partial_{x_i}^{(m)}$, $i=1,\ldots, r$, $m\geq 0$,
			which ``behave'' like
			$\frac{1}{m!}\partial^{m}/\partial x_i^m$. If
			$\overline{U}\cap (\overline{X}\setminus X)$ is defined
			by $x_1\cdot\ldots\cdot x_r$, then the subring
			$\mathscr{D}_{\overline{X}/k}(\log
			\overline{X}\setminus X)|_{\overline{U}}$ is spanned
			by $\delta^{(m)}_{x_i}:=x_i^m\partial_{x_i}^{(m)}$,
			$i=1,\ldots, r$, $\partial_{x_i}^{(m)}$, $i>r, m\geq
			0$. 
			
			This ring can also be defined intrinsically as the ring of differential operators
			associated with a morphism of $\log$-schemes, which,
			for example,  is
			studied in \cite{Montagnon}.
		\item If $(X,\overline{X})$ is a good partial
			compactification, then a stratified bundle $E\in \Strat(X)$ is called
			\emph{$(X,\overline{X})$-regular singular} if $E$
			extends to an $\mathcal{O}_{\overline{X}}$-torsion
			free, $\mathcal{O}_{\overline{X}}$-coherent
			$\mathscr{D}_{\overline{X}/k}(\log
			\overline{X}\setminus X)$-module. 

			We write $\Strat^{\rs}( (X,\overline{X}))$ for the
			full subcategory of $\Strat(X)$ with objects the
			$(X,\overline{X})$-regular singular stratified
			bundles.
		\item A stratified bundle $E$ is called \emph{regular
			singular} if $E$ is $(X,\overline{X})$-regular
			singular for every good partial compactification
			$(X,\overline{X})$ of
			$X$.

		We write $\Strat^{\rs}(X)$ for the full subcategory of
			$\Strat(X)$ with objects the regular singular stratified
			bundles.
	\end{enumerate}
\end{Definition}
\begin{Remark}The notion of $(X,\overline{X})$-regular singularity for stratified bundles
first appeared (to the author's knowledge)
in \cite{Gieseker/FlatBundles} for good compactifications $(X,\overline{X})$.
Gieseker attributes some of the ideas used in [loc.~cit.] to
Katz.  
In \cite{Kindler/FiniteBundles} and \cite{Kindler/thesis} this notion of
regular singularity is extended to varieties for which resolution of
singularities is not known to hold, and its connection with tame ramification
is studied.
\end{Remark}

For the purpose of this article, the following facts are of importance:
\begin{Proposition}\label{prop:rsfacts}
	Let $X$ be a smooth, separated, finite type $k$-scheme. 
	\begin{enumerate}
		\item\label{rsfacts1} If $X$ admits a good compactification
			$(X,\overline{X})$, then a stratified bundle $E$ on
			$X$ is regular singular if and only if it is
			$(X,\overline{X})$-regular singular.
		\item\label{rsfacts2} If $E$ is a stratified bundle of rank $1$, then $E$ is
			regular singular.
		\item\label{rsfacts3}If
	$\iota:\Strat^{\rs}(X)\rightarrow \Strat(X)$ denotes the inclusion
	functor, then for every object $E \in \Strat^{\rs}(X)$, $\iota$
	induces an equivalence
	$\left<E\right>_{\otimes}\xrightarrow{\cong}\left<\iota(E)\right>_{\otimes}$.
	\end{enumerate}
\end{Proposition}
\begin{proof}
	Statement \ref{rsfacts1} is \cite[Prop.~7.5]{Kindler/FiniteBundles},
	\ref{rsfacts3} is \cite[Prop.~4.5]{Kindler/FiniteBundles},
	and \ref{rsfacts2} is \cite[Lemma 3.12]{Gieseker/FlatBundles}. We
	recall some arguments for \ref{rsfacts2} from [loc.~cit.] for convenience:
	Let $A$ be a finite type $k$-algebra with
	coordinates $x_1,\ldots, x_n$, and $M$ a free rank $1$ module
	over $A[x_1^{-1}\cdot\ldots\cdot x_r^{-1}]$, $1\leq
	r\leq n$. Assume that $M$ carries an action of
	$\mathscr{D}_{A[x_1^{-1},\ldots,x_r^{-1}]/k}$, and that $e$ is a basis
	of $M$.  With the notation from Definition \ref{defn:diffops},
	\ref{item:diffops}, it suffices to show that $\delta_{x_i}^{(m)}(e)\subset
	eA$ for every $m>0$ and every $1\leq i \leq r$. 
	By Theorem \ref{thm:cartier}, for $N\in \N$ there exists a nonzero section
	$s\in A$ such that
	$\partial_{x_i}^{(m)}(se)=0$ for all $m\leq p^N$. But then in $\Frac
	A$ we have
	\[\delta^{(m)}_{x_i}(e)=\delta^{(m)}_{x_i}(s^{-1}se)=\delta^{(m)}_{x_i}(s^{-1})se\]
	and $\delta^{(m)}_{x_i}(s^{-1})s\in A$.
\end{proof}
\begin{Remark}Recall that the analogous statement to \ref{rsfacts2} in
	characteristic $0$ is false.
\end{Remark}

We recall the definition of the monodromy group of a stratified bundle:
\begin{Definition}
	If $E\in \Strat(X)$, then
	$\left<E\right>_{\otimes}$ is the Tannakian subcategory
	category of  $\Strat(X)$ generated by $E$. If
	$\omega:\left<E\right>_{\otimes}\rightarrow \Vectf_k$
	is a fiber functor, then the associated $k$-group
	scheme is called the \emph{monodromy group of $E$
	(with respect to $\omega$)}. By \cite{DosSantos}, the monodromy group
	of a stratified bundle is always smooth. 
	
	Note that if $E$ is
	regular singular, then by Proposition
	\ref{prop:rsfacts}, \ref{rsfacts2}, it makes no difference whether we
	compute the monodromy group of $E$ as an object of $\Strat^{\rs}(X)$
	or $\Strat(X)$. 
	
	Moreover, note that since $k$ is
	algebraically closed, the isomorphism class of the
	monodromy group does not depend on the choice of the
	fiber functor.
\end{Definition}

To finish this subsection, we briefly recall the notion of exponents of
regular singular stratified bundles. For more details see
\cite{Kindler/FiniteBundles} or \cite{Kindler/thesis}.

\begin{Proposition}\label{prop:exponents}Let $(X,\overline{X})$ be a good partial compactification
	of $X$, such that $D:=\overline{X}\setminus X$ is irreducible. Let
	$\overline{E}$ be a	$\mathscr{D}_{\overline{X}/k}(\log
	\overline{X}\setminus X)$-module, which is
	$\mathcal{O}_{\overline{X}}$-locally free of finite rank.
	Then the following are true:
	\begin{enumerate}
		\item \emph{(\cite[Lemma 3.8]{Gieseker/FlatBundles})} There exists a decomposition
			\[ \overline{E}|_D=\bigoplus_{\alpha\in \Z_p} F_{\alpha},\]
			with the property that if $x_1,\ldots, x_n$ are local
			coordinates around the generic point of $D$, such that
			$D=(x_1)$, then $e+x_1\overline{E}\in F_{\alpha}$ if and only
			if 
			\[\delta_{x_1}^{(m)}(e)=\binom{\alpha}{m}e+x_1\overline{E}.\]
			For the definition of $\delta_{x_1}^{(m)}$, see
			Definition \ref{defn:diffops}, \ref{item:diffops}.
			Write $\Exp(\overline{E})\subset \Z_p$ for the set of
			those $\alpha\in \Z_p$ for which $\rank F_\alpha\neq
			0$.
		\item \emph{(\cite[Prop.~4.12]{Kindler/FiniteBundles})} If $\overline{E}'$ is a second
			$\mathcal{O}_{\overline{X}}$-locally free, finite
			rank, $\mathscr{D}_{\overline{X}/k}(\log
			\overline{X}\setminus X)$-module, such that the
			stratified bundles $\overline{E}|_X$ and
			$\overline{E}'|_X$ are isomorphic, then  the sets
			$\Exp(\overline{E})$ and $\Exp(\overline{E'})$ have
			the same image in $\Z_p/\Z$.
	\end{enumerate}
\end{Proposition}
\begin{Definition}[{\cite[Def.~4.13]{Kindler/FiniteBundles}}]
	If $(X,\overline{X})$ is a good partial compactification, $E$ an
	$(X,\overline{X})$-regular singular stratified bundle and 
	$D$ an irreducible component of $\overline{X}\setminus X$, then we
	define the \emph{exponents of $E$ along $D$} as follows: Let
	$\overline{U}$ be a sufficiently small open neighborhood of the
	generic point of $D$, such that there exists an
	$\mathcal{O}_{\overline{U}}$-locally free
	$\mathscr{D}_{\overline{U}/k}(\log \overline{U}\setminus X)$-module
	extending $E|_{\overline{U}\cap X}$. Then the set of exponents of $E$
	along $D$ is the image of the set $\Exp(\overline{E})$ in $\Z_p/\Z$ from
	Proposition \ref{prop:exponents}. This construction is well-defined by
	Proposition \ref{prop:exponents}.
\end{Definition}

Having all exponents equal to $0 \mod \Z$ is a stronger condition in
characteristic $p>0$ than in characteristic $0$. In particular there are no
regular singular stratified bundles with nontrivial ``nilpotent residues'':
\begin{Proposition}[{\cite[Cor.~5.4]{Kindler/FiniteBundles}}]\label{prop:nonilpotent}
	Let $(X,\overline{X})$ be a good partial compactification and $E$ an
	$(X,\overline{X})$-regular singular stratified bundle. If the
	exponents of $E$ along all components of $\overline{X}\setminus X$
	are $0$ in $\Z_p/\Z$, then there exists a stratified bundle
	$\overline{E}$ on $\overline{X}$ extending $E$.
\end{Proposition}

%
%
%
\section{Special case I: Stratified line bundles}\label{sec:linebundles}
We continue to denote by $k$ an algebraically closed field of characteristic
$p>0$, and by $X$ a smooth, separated, connected,  finite type $k$-scheme.

We start with a group theoretic definition.
\begin{Definition}
	If $G$ is an abelian group, write $(G)_p$ for the projective system 
	\begin{equation}\label{projSystem} G\xleftarrow{\cdot p}
		G\xleftarrow{\cdot p} \ldots.\end{equation}
	This defines an exact functor from the category of abelian groups into
	the category of pro-systems of abelian groups. We write $\plim
	G:=\varprojlim ((G)_p)$, and $R^1\plim G:=
	R^1\varprojlim((G)_p)$.
	This construction defines a left exact functor $\plim$ from the category
	of abelian groups to itself, and for a short exact sequence
	\[0\rightarrow A\rightarrow B \rightarrow C\rightarrow 0\] of abelian
	groups, we have the long exact squence
	\[0\rightarrow\plim A\rightarrow \plim B\rightarrow
		\plim C\rightarrow R^1\plim A\rightarrow
		R^1\plim B\rightarrow R^1\plim C \rightarrow
		0.\]

\end{Definition}

\begin{Definition}
	We write $\Pic^{\Strat}(X)$ for the set of isomorphism classes of
	stratified bundles of rank $1$. The tensor product of stratified
	bundles gives this set an abelian group structure.
\end{Definition}

By Theorem \ref{thm:cartier} we can associate with every stratified line
bundle a sequence of elements $L_n\in \Pic X$, such that $L_{n+1}^p=L_n$. This
gives a group homomorphism $\Pic^{\Strat}(X)\rightarrow \varprojlim_{p} \Pic
X$. 

\begin{Definition}Denote by $\mathbb{I}(X)$ the subgroup of $\Pic^{\Strat}(X)$
of isomorphism classes of stratified line bundles $(L_n,\sigma_n)$ with
$L_n\cong\mathcal{O}_X$ for all $n$. 
\end{Definition}

We clearly have the following:

\begin{Proposition}\label{prop:SESpicstrat}
	The morphisms described above fit in a functorial short exact sequence
	\[ 
		0\rightarrow \mathbb{I}(X)\rightarrow \Pic^{\Strat}(X)\rightarrow
		\varprojlim_{p}\Pic(X)\rightarrow 0.
		\]
\end{Proposition}

Moreover, it is not difficult to describe the group $\mathbb{I}(X)$
concretely: 

\begin{Proposition}\label{prop:descriptionofI}
	There is a functorial exact sequence
	\begin{equation*}
		\xymatrix@C=.5cm{ 
			0\ar[r]&k^\times\ar[r]&H^0(X,\mathcal{O}_X^\times)\ar[r]^--{\Delta}&\varprojlim_n
			H^0(X,\mathcal{O}_X^{\times})/H^0(X,\mathcal{O}_{X}^{\times})^{p^n}\ar[r]&
			\mathbb{I}(X)\ar[r]&0
			 }
		 \end{equation*}
		 where the morphism $k^\times \rightarrow
		 H^0(X,\mathcal{O}_X^\times)$ is the canonical inclusion, and
		 $\Delta$ the diagonal map. 
	 \end{Proposition}
\begin{proof}
	Since the base field $k$ is algebraically closed of characteristic
	$p>0$, the sequence is exact at
	$H^0(X,\mathcal{O}_X^\times)$. Indeed,
	$H^0(X,\mathcal{O}_X^\times)/k^\times$ is a finitely generated free abelian
	group, because of the exact sequence
        \begin{equation*}
        	\xymatrix@C=.5cm{ 
        		0\ar[r]&
			H^0(\overline{X},\mathcal{O}_{\overline{X}}^\times)\ar[r]&H^0(X,\mathcal{O}_X^\times)\ar[r]&\bigoplus_i
			D_i
			\mathbb{\Z}\ar[r]&\Cl(\overline{X})\ar[r]&\Cl(X)\ar[r]&0,
       		 }
        	 \end{equation*}
		 where $\overline{X}$ is any normal compactification of $X$, and
		 $D_i$ the irreducible components of $\overline{X}\setminus
		 X$. This shows that the kernel of $\Delta$ is precisely
		 $k^\times$.
		 
		 Thus we only need to construct the morphism
	\[\phi:\coker(\Delta)\rightarrow
		\mathbb{I}(X),\]
		and show that it is a natural isomorphism. 

		Given an element $\alpha:=(\alpha_1,\alpha_2,\ldots)\in \varprojlim_n
		H^0(X,\mathcal{O}_X^\times)/H^0(X,\mathcal{O}_X^\times)^{p^n}$,
		take any sequence of lifts $a:=(a_1,a_2,\ldots) \in
		H^0(X,\mathcal{O}_X^\times)^{\N}$, and define the stratified bundle
		$\Psi(a):=(\Psi_n(a),\sigma^a_n)_{n\geq 0}$ by setting 
		$\Psi(a)_n=\mathcal{O}_X$, and by defining the transition isomorphisms
		$\sigma_n^a:\Psi(a)_n\rightarrow F^*\Psi(a)_{n+1}$ as follows:
		$\sigma_0^a$
		is multiplication by $a_1$, and
		$\sigma^a_n=(a_{n+1}a_n^{-1})^{1/p^n}$, $n>0$. This works, since
		by definition $a_{n+1}\equiv a_n \mod p^{n}$. We have to check
		that this construction gives a well-defined map $\Psi$: If we pick a different sequence of lifts
		$a':=(a_1',a_2',\ldots)$ of
		$\alpha=(\alpha_1,\alpha_2,\ldots)$, then the resulting stratified line
		bundle $\Psi(a')$ is isomorphic to $\Psi(a)$. Indeed, if
		$b_n^{p^n}=a'_na_n^{-1}$, then the sequence of isomorphisms
		$\phi_0=\id$, 
		$\phi_n:\Psi(a)_n\xrightarrow{\cdot b_n}\Psi(a')_n$, $n>1$,  defines an
		isomorphism of stratified bundles. We write $\Psi(\alpha)$ for
		the isomorphism class of the stratified bundle $\Psi(a)$

		It is readily checked that this map is surjective: If the stratified bundle
		$L=(L_n,\sigma_n)_{n\geq 0}$ is given by $L_n=\mathcal{O}_X$
		and the transition morphisms $\sigma_n$,
		then $\sigma_n\in H^0(X,\mathcal{O}^{\times}_X)$, and
		$L=\Psi(
		(\sigma_0,\sigma_0\sigma_1^p,\sigma_0\sigma_1^p\sigma_2^{p^2},\ldots) )$.

		To compute the kernel, note that a stratification $L$ on
		$\mathcal{O}_X$ which is given by transition morphisms
		$\sigma_n$, is trivial if and only if there exists a sequence
		$\phi_0,\phi_1,\ldots \in H^0(X,\mathcal{O}_X^\times)$ such
		that $\sigma_n\phi_{n+1}^p=\phi_n$, $n>0$. By recursion, this
		means that 
		\[\phi_0=\sigma_0\sigma_1^p\cdot\ldots\cdot
			\sigma_n^{p^n}\phi_{n+1}^{p^{n+1}}\]
			for every $n$. In other words: $L=\Psi(\Delta(\phi_0))$,
			so
			the kernel of $\Psi$ is precisely
			$H^0(X,\mathcal{O}_X^\times)$. 
			This completes the proof.

%
%
%
	\end{proof}


	\begin{Remark}\label{remark:completion}
	\begin{enumerate}[label={(\alph*)}]
		\item The description of $\mathbb{I}(X)$ in Proposition
			\ref{prop:descriptionofI} was inspired by a similar
			description for stratified bundles on $k\llparen
			t\rrparen$ in 
			\cite{Matzat}. 
		\item\label{remark:completion1} Note that the abelian group $\varprojlim_n
			H^0(X,\mathcal{O}_X^\times)/H^0(X,\mathcal{O}_X^\times)^{p^n}$
			is just the $p$-adic completion of the free finitely
			generated abelian group
			$H^0(X,\mathcal{O}_X^\times)/k^\times$, and the map
			$\Delta$ induces the canonical map from
			$H^0(X,\mathcal{O}_X^\times)/k^\times$ into its
			$p$-adic completion.
	\end{enumerate}
		\end{Remark}
	\begin{Corollary}\label{cor:pdivisible}
		The group $\mathbb{I}(X)$ is trivial if and only if
		$H^0(X,\mathcal{O}_X^\times)=k^\times$.
	\end{Corollary}
	\begin{proof}
		This follows from Proposition \ref{prop:descriptionofI} and Remark \ref{remark:completion},
		\ref{remark:completion1}, because the map from $\Z^r$ into its
		$p$-adic completion is surjective if and only if $r=0$.
	\end{proof}

Corollary \ref{cor:pdivisible} allows us to exhibit a class of examples of
varieties $X$ such that $\mathbb{I}(X)$ is trivial:

\begin{Proposition}\label{prop:globalunits}
	If $k$ is an algebraically closed field, and if
	$X$ is a connected normal $k$-scheme of finite type, such that the maximal
	abelian pro-$\ell$-quotient
	$\pi_1(X)^{\ab,(\ell)}$ is trivial for some $\ell\neq \Char(k)$, then
	$H^0(X,\mathcal{O}_X^\times) = k^\times$. In particular, if $k$ has
	positive characteristic, then $\mathbb{I}(X)=0$. 
\end{Proposition}
\begin{proof} This argument is due to H\'el\`ene Esnault.
	Assume $f\in H^0(X,\mathcal{O}_X^\times)\setminus k^\times$. Then $f$ induces a
	dominant morphism $f':X\rightarrow \mathbb{G}_{m,k}\cong
	\mathbb{A}^1_k\setminus\left\{ 0 \right\}$, as $f'$ is given by the map $k[x^{\pm 1}]\rightarrow H^0(X,\mathcal{O}_X)$,
	$x\mapsto f$, which is injective if and only if $f$ is transcendental over $k$.
	Thus $f'$ induces an \emph{open} morphism $\pi_1(X)\rightarrow
	\pi_1(\mathbb{G}_{m,k})$, see e.g. \cite[Lemma 4.2.10]{Stix/dissertation}. 
	But under our assumption, the maximal abelian pro-$\ell$-quotient of the
	image of this morphism is trivial, so the image of $\pi_1(X)$ cannot have finite index in the
	group $\pi_1(\mathbb{G}_{m,k})$, as in fact
	$\pi_1(\mathbb{G}_{m,k})^{(\ell)}\cong
	\widehat{\Z}^{(\ell)}=
	\Z_\ell$.
        %
\end{proof}
\begin{Remark}\label{rem:affinevars}
	We can modify the argument of Proposition \ref{prop:globalunits}
	slightly to obtain: 
	If $k$ has positive characteristic $p$, and
	$\pi_1(X)^{\ab,(p)}=0$, then $H^0(X,\mathcal{O}_X)=k$.
	A consequence is a proof of the folklore fact that over a
	field $k$ of positive characteristic, unlike in characteristic $0$, no
	normal, connected, affine, finite type $k$-scheme $X$ of positive dimension is simply-connected, or even
	has $\pi_1(X)^{\ab,(p)}=0$.

	Indeed, if $f\in H^0(X,{\mathcal{O}_X})$ is nonconstant, then $f$
	induces a dominant morphism $X\rightarrow \A^1_k$ and hence an open
	morphism $\pi_1^{\ab,(p)}(X)\rightarrow \pi_1^{\ab,(p)}(\A^1_k)$. But
	by
	\cite[1.4.3, 1.4.4]{Katz/LocalToGlobal} we have
	$H^2(\pi_1(\A^1_k),\F_p)=0$, so $\pi_1(\A^1_k)^{(p)}$
	is free pro-$p$ of rank $\dim_{\F_p}H^1(\A_k^1,\F_p)=\#k$. Thus the
	image of $\pi^{\ab,(p)}_1(X)$ in this group can only have finite index if
	$\pi_1(X)^{\ab,(p)}\neq 0$.

\end{Remark}

Recall that a \emph{good compactification} of a smooth $k$-scheme $X$ is a
proper, smooth, finite type $k$-scheme $\overline{X}$, with a dominant open
immersion $X\hookrightarrow \overline{X}$, such that $\overline{X}\setminus X$
is the support of a strict normal crossings divisor.

The main goal of this section is to prove the following theorem:
\begin{Theorem}\label{thm:rank1}Let $X,Y$ be smooth, separated, finite type
	$k$-schemes, and
	let $f:Y\rightarrow X$ be a morphism such that the following conditions are
	satisified:
	\begin{enumerate}
		\item There exist good compactifications $\overline{X}$ and
			$\overline{Y}$ of $X$ and $Y$, such that $f$ extends to a morphism
			$\bar{f}:\overline{Y}\rightarrow \overline{X}$.
		\item\label{prop:projlim:assumption2}  $f$ induces an isomorphism
	\begin{equation}
		\label{eq:iso1}f_*:\pi_1^{\ab,(p')}(Y)\rightarrow
		\pi_1^{\ab,(p')}\pi_1(X),
	\end{equation} 
	\item\label{prop:projlim:assumption3}  The cokernel of the morphism
	\begin{equation}\label{eq:iso2}
		H^0(X,\mathcal{O}_X^\times)\rightarrow H^0(Y,\mathcal{O}_{Y}^\times).
	\end{equation} 
	induced by $f$ is a $p$-group.
	\end{enumerate}
			Then pull-back along $f$ induces an isomorphism
			\[\Pic^{\Strat}(X)\xrightarrow{\cong} \Pic^{\Strat}(Y).\]
\end{Theorem}
\begin{Remark}\label{rem:globalunits}
	\begin{enumerate}[label={(\alph*)}]
		\item The morphism
			\eqref{eq:iso2} is an isomorphism in the
			following two situations:
			\begin{enumerate}[label=(\roman*)]
				\item $X,Y$ proper over $k$.
				\item $\pi_1^{\ab,(p')}(Y)=\pi_1^{\ab,(p')}(X)=1$
					(e.g.~\eqref{eq:iso1} is an
					isomorphism and $X=\Spec(k)$).
					See Proposition
					\ref{prop:globalunits}.
			\end{enumerate}
			Thus in these two cases, the theorem establishes that $f$
			induces an isomorphism of abelian groups
			$\Pic^{\Strat}(X)\rightarrow \Pic^{\Strat}(Y)$, if
		 	\eqref{eq:iso1} is an isomorphism.
		\item\label{rem:injectivity} If \eqref{eq:iso1} is an isomorphism, then
			\eqref{eq:iso2} is always injective. Indeed, if
			$\alpha\in H^0(X,\mathcal{O}_X^\times)$ is a global
			unit, then $\alpha$ induces a morphism
			$X\rightarrow \G_m$, which is dominant if and only if
			$\alpha$ is non-constant. Hence the induced continuous morphism
			$\pi_1^{\ab, (p')}(X)\rightarrow
			\widehat{\Z}^{(p')}$ is open if and only if $\alpha$
			is non-constant. If \eqref{eq:iso1} is an isomorphism,
			then $f^*\alpha$ is non-constant whenever $\alpha$ is.
			In particular, 
			it follows that \eqref{eq:iso2} is injective.
		\item Clearly \eqref{eq:iso2} is not necessarily an
			isomorphism, even if \eqref{eq:iso1} is: Just take
			the purely inseparable morphism $\G_m\rightarrow \G_m$
			defined by taking the $p$-th root of a coordinate on
			$\G_m$.
		\item The question whether \eqref{eq:iso2} has a $p$-group as
			cokernel whenever \eqref{eq:iso1} is an isomorphism, 
			belongs to the area of Grothendieck's anabelian
			geometry: The cokernel of \eqref{eq:iso2} is finitely
			generated by \cite[Lemme 1]{Kahn/GroupeDesClasses}, so
			we can split off the $p$-power torsion. If $\alpha\in
			H^0(Y,\mathcal{O}_Y^{\times})$
			is a global unit of order prime to $p$ in the cokernel
			of \eqref{eq:iso2}, and if \eqref{eq:iso1} is an isomorphism,
			then $\alpha$ comes from a global unit on $X$ if and
			only if the induced morphism on fundamental groups
			\[\pi_1^{\ab,(p')}(X)\xrightarrow{\cong}\pi_1^{\ab,(p')}(Y)\xrightarrow{\alpha_*}
			\pi_1^{\ab,(p')}(\G_m)=\widehat{\Z}^{(p')}\]
			is induced by a morphism of $k$-schemes $X\rightarrow
			\G_m$. The author does not know whether the condition
			that \eqref{eq:iso1} is an
			isomorphism always implies that the cokernel of \eqref{eq:iso2}  is a
			$p$-group.
	\end{enumerate}
\end{Remark}

We need a few lemmas to prepare the proof of Theorem \ref{thm:rank1}.

\begin{Lemma}\label{lemma:R1} If $G$ is a finitely generated abelian group,
	then the functor $\plim$ is naturally isomorphic to the
	functor which assigns to $G$ its subgroup $G[p']$ of torsion elements
	of order prime to $p$.

	If $G$ is a finite abelian group or an abelian group (not necessarily
	finitely generated) on which multiplication by $p$ is surjective, then
	$R^1\plim G=0$.  
\end{Lemma} 
	\begin{proof} In a finitely
		generated abelian group, an element is infinitely $p$-divisible if and
		only if it is torsion of order prime to $p$.  Moreover, an
		element $x\in G[p']$ admits a unique $p$-th root in $G[p']$.
		It follows that the map $\plim G\rightarrow G[p']$,
		$(x_1,x_2,\ldots)\mapsto x_1$ is an isomorphism.

	For the second claim, if $G$ is finite or if multiplication by $p$ on
	$G$ is surjective, then the projective system \eqref{projSystem}
	satisfies the Mittag-Leffler condition, so $R^1\plim G=0$.
\end{proof}

\begin{Lemma}\label{lemma:groupTheory} Consider the following morphism of
	short exact sequences of abelian groups \begin{equation*} \xymatrix{
			0\ar[r]&D_1\ar[r]\ar[d]_{f}&G_1\ar[r]\ar[d]_g&F_1\ar[r]\ar[d]_h&0\\
			0\ar[r]&D_2\ar[r]&G_2\ar[r]&F_2\ar[r]&0 }
		\end{equation*} with $F_1,F_2$ finitely generated, and such
		that multiplication by $p$ on $D_1$, $D_2$ is surjective
		(e.g.~$D_1,D_2$ divisible).

	Assume that the following conditions are satisfied:
	\begin{enumerate}[label={\emph{(\alph*)}},ref={(\alph*)}] 
		\item\label{prop1} $f$ is surjective with finite kernel.  \item\label{prop2}
			$\ker(g)$ contains no torsion elements of order prime
			to $p$.  
		\item\label{prop3} The restriction $h|_{F_1[p']}:
			F_1[p']\rightarrow F_2[p']$ is surjective.
	\end{enumerate} 
	Then the induced morphism $\plim
	G_1\rightarrow \plim G_2$ is an isomorphism.  
\end{Lemma}
\begin{proof} 
	Since $R^1\plim D_i=0$ for $i=1,2$ by Lemma
	\ref{lemma:R1}, we get the following morphism of short exact
	sequences: \[ \xymatrix{ 0\ar[r]&\plim
		D_1\ar[r]\ar[d]_{\plim f}&\plim
		G_1\ar[r]\ar[d]_{\plim g}&\plim
		F_1\ar[r]\ar[d]_{\plim h}&0\\
		0\ar[r]&\plim D_2\ar[r]&\plim
	G_2\ar[r]&\varprojlim F_2\ar[r]&0 } \] By the
	left-exactness of $\plim$, we have
	$\plim(\ker g)=\ker(\plim g)$.  But by
	assumption \ref{prop1}, $\ker g$ is an extension of
	two finitely generated groups, and hence finitely
	generated itself.  It has no prime-to-$p$ torsion by
	\ref{prop2}, so by Lemma \ref{lemma:R1},
	$\plim(\ker g)=0$.  Thus $\plim g$ is
	injective.

	Next, by \ref{prop1}, $f$ is surjective with finite kernel,
	which implies that $\plim f$ is surjective, since
	$R^1\plim(\ker(f))=0$ by Lemma \ref{lemma:R1}.

	Finally, by Lemma \ref{lemma:R1} we know that $\plim
	h=h|_{F_1[p']}:F_1[p']\rightarrow F_2[p']$.  This morphism is
	surjective by \ref{prop3}.  It follows that $\plim g$
	is also surjective, which completes the proof.  
\end{proof}
\begin{Lemma}\label{lemma:pi1compactification} If $f:Y\rightarrow X$
	is a morphism of connected, smooth, separated, finite type
	$k$-schemes, such that $f$ induces an isomorphism
	$\pi_1(Y)^{\ab,(p')}\xrightarrow{\cong}\pi_1(X)^{\ab,(p')}$, and
	if $\overline{X},\overline{Y}$ are smooth compactifications of
	$X,Y$, such that $f$ extends to
	$\bar{f}:\overline{Y}\rightarrow \overline{X}$, then the map
	\[ 
		\pi_1(\overline{Y})^{\ab,(p')}\rightarrow
		\pi_1(\overline{X})^{\ab,(p')} 
		\] 
	induced by $\bar{f}$
	is surjective with finite kernel.  
\end{Lemma}
\begin{proof} The surjectivity is clear, as
	$\pi_1(X)^{\ab,(p')}$ and
	$\pi_1(Y)^{\ab,(p')}$ surject onto
	$\pi_1(\overline{X})^{\ab,(p')}$ and
	$\pi_1(\overline{Y})^{\ab,(p')}$,
	respectively.  We use the theory of the
	Albanese variety: After choosing a base point
	$x\in X(k)$, there exists a unique
	semi-abelian variety $\Alb_X$, together with a
	map $\alb_{X,x}:X\rightarrow \Alb_X$, such
	that $\alb_{X,x}(x)=0$, and such that any map
	$g:X\rightarrow A$ from $X$ to a semi-abelian
	variety $A$ with $g(x)=0$ factors uniquely
	through $\alb_{X,x}$.  For more details, see
	e.g.~\cite{Szamuely/Albanese}.  The Albanese
	variety classifies abelian coverings of $X$ in
	the following sense: For $\ell$ a prime
	different from $p$, there exists a canonical
	isomorphism 
	\[
		\hom(H^1(X,\Z_\ell),\Q_\ell/\Z_\ell)\cong
		\Alb_X(k)\{\ell\}, 
		\] 
	see \cite[Cor.~4.3]{Szamuely/Albanese}, where $\Alb_X(k)\{\ell\}$ is
	the the subgroup of $\Alb_X(k)$ of all $\ell$-power torsion elements.
	As a semi-abelian variety, $\Alb_X$ can be written uniquely as an
	extension of an abelian variety by a torus.  The unique abelian
	variety appearing in this description of $\Alb_X$ is canonically
	isomorphic to $\Alb_{\overline{X}}$. 

	By Chevalley's structure theorem for algebraic groups, the connected
	component of the origin (with its reduced structure) $K^0_{\red}$ of
	the kernel of the morphism of group schemes $\Alb_Y\rightarrow \Alb_X$
	is a semi-abelian variety, and since every map from an abelian variety
	into an affine variety is constant, it follows that the unique abelian
	variety quotient of $K^0_{\red}$ is precisely the connected component
	of the origin (with its reduced structure) of the kernel of
	$\Alb_{\overline{Y}}\rightarrow \Alb_{\overline{X}}$.

	With that in mind, we see that the assumption that
	$\pi_1(Y)^{\ab,(p')}\xrightarrow{\cong}\pi_1(X)^{\ab,(p')}$ is an
	isomorphism implies that the map $\Alb_Y(k)\rightarrow \Alb_X(k)$ is
	an isomorphism on prime-to-$p$ torsion, and hence $K^0_{\red}$ is the
	trivial semi-abelian variety.  Then
	$\ker(\Alb_{\overline{Y}}\rightarrow \Alb_{\overline{Y}})^0_{\red}$ is
	the trivial abelian variety, which shows that the kernel of
	$\Alb_{\overline{Y}}(k)\rightarrow \Alb_{\overline{X}}(k)$ is a finite
	group. 

	Writing $T_{\ell}$ for the Tate module, we see that if $\ell$
	is a prime different from $p$, then $T_\ell
	\Alb_{\overline{Y}}\rightarrow T_\ell\Alb_{\overline{X}}$ is
	injective.  Finally, since $\overline{X}$ and $\overline{Y}$
	are proper and smooth, we get a morphism of short exact
	sequences
	\begin{equation*} \xymatrix{
			0\ar[r]&\hom(\NS(\overline{Y})\{\ell\},\Q_\ell/\Z_\ell)\ar[r]\ar[d]&\pi_1(\overline{Y})^{\ab,(\ell)}\ar[r]\ar[d]&T_\ell
			\Alb_{\overline{Y}}\ar[r]\ar@{^{(}->}[d]&
			0\\
			0\ar[r]&\hom(\NS(\overline{X})\{\ell\},\Q_\ell/\Z_\ell)\ar[r]&\pi_1(\overline{X})^{\ab,(\ell)}\ar[r]&T_\ell
			\Alb_{\overline{X}}\ar[r]& 0.
		} 
	\end{equation*} 
	The groups on the left are finite groups, so we deduce that
	$\pi_1(\overline{Y})^{\ab,(\ell)}\rightarrow\pi_1(\overline{X})^{\ab,(\ell)}$
	always has a finite kernel, and is injective for all but
	finitely many $\ell$.  This implies that the kernel of
	$\pi_1(\overline{Y})^{\ab,(p')}\rightarrow
	\pi_1(\overline{X})^{\ab,(p')}$ is finite.  
\end{proof}

\begin{Proposition}\label{prop:projlim} 
	With the notations and assumptions of Theorem \ref{thm:rank1}, 
			pullback along $f$ induces an isomorphism
			\[
				\plim \Pic(X)
				\xrightarrow{\cong}\plim
				\Pic(Y).
				\] 
\end{Proposition} 
\begin{proof}
	By
	Lemma \ref{lemma:pi1compactification} the induced morphism
	$\pi_1(\overline{Y})^{\ab,p'}\rightarrow \pi_1(\overline{X})^{\ab,p'}$
	is surjective with finite kernel.  Define
	$K_X:=\ker(\Pic(\overline{X})\rightarrow \Pic(X))$ and
	$K_Y:=\ker(\Pic(\overline{Y})\rightarrow \Pic(Y))$.  We get a
	commutative diagram
	\begin{equation*}
		\begin{split}
			\xymatrix{
				0\ar[r]&K_X\ar[d]\ar[r]&\ar[d]\Pic(\overline{X})\ar[r]&\ar[d]\Pic(X)\ar[r]&0\\
				0\ar[r]&K_Y\ar[r]&\Pic(\overline{Y})\ar[r]&\Pic(Y)\ar[r]&0
			} 
		\end{split}
	\end{equation*} 
	with $K_X, K_Y$ finitely generated abelian groups.  We
	have a second exact sequence
	\[ 0\rightarrow \Pic^0(\overline{X})\rightarrow \Pic(\overline{X})\rightarrow \NS(\overline{X})\rightarrow
		0 
		\] 
	with $\NS(\overline{X})$ a finitely generated group, according to,
	e.g., \cite[Thm.~3]{Kahn/GroupeDesClasses}.  Moreover,
	$\Pic^0(\overline{X})$ is the set of $k$-points of an abelian variety,
	so $\Pic^{0}(\overline{X})$ is a divisible abelian group.

	Define $\Pic^0(X)$ as the image of $\Pic^0(\overline{X})$ in $\Pic(X)$,
	and $\NS(X)$ as $\Pic(X)/\Pic^0(X)$.  Then $\Pic^0(X)$ still divisible,
	and $\NS(X)$ is still finitely generated.  Pullback along $f$ induces
	a morphism $\Pic^0(X)\rightarrow \Pic^0(Y)$, and hence a morphism
	$\NS(X)\rightarrow \NS(Y)$.

	We obtain the following commutative diagram with exact
	rows:
	\begin{equation}\label{diag:mainpic0}
		\begin{split}
		\xymatrix{
			0\ar[r]&\Pic^0(X)\ar[d]\ar[r]&\Pic(X)\ar[r]\ar[d]&\NS(X)\ar[r]\ar[d]&0\\
			0\ar[r]&\Pic^0(Y)
			\ar[r]&\Pic(Y)\ar[r]&\NS(Y)\ar[r]&0
		} 
	\end{split}
	\end{equation} 
	We are now in the situation of Lemma \ref{lemma:groupTheory}, and
	check that the conditions \ref{prop1}, \ref{prop2} and \ref{prop3}
	from the lemma are satisfied for diagram \eqref{diag:mainpic0}.

	For Lemma \ref{lemma:groupTheory}, \ref{prop1} we have to show that
	$\Pic^0(X)\rightarrow \Pic^0(Y)$ is surjective with finite kernel.

	For every $n$ prime to $p$, Kummer theory shows that there a morphism
	of short exact sequences of abelian groups
	\begin{equation}\label{diag:pic0}
		\xymatrix@C=.3cm{
			0\ar[r]&\ar[d]H^0(X,\mathcal{O}_X^{\times})/H^0(X,\mathcal{O}_X^{\times})^n\ar[r]
			&\ar[d]
			\Hom(\pi_1^{\ab,(p')}(X),\Z/n\Z)\ar[r]
			&\ar[d]\Pic(X)[n]\ar[r]
			&
			0\\
			0\ar[r]&H^0(Y,\mathcal{O}_Y^{\times})/H^0(Y,\mathcal{O}_Y^{\times})^n\ar[r]
			&
			\Hom(\pi_1^{\ab,(p')}(Y),\Z/n\Z)\ar[r]
			&\Pic(Y)[n]\ar[r]
		& 0 }
	\end{equation} 
	where the two left vertical arrows are isomorphisms by the
	assumptions \ref{prop:projlim:assumption2} and
	\ref{prop:projlim:assumption3} of Theorem \ref{thm:rank1}, and hence so is the third.  Since by Lemma
	\ref{lemma:pi1compactification} the morphism
	$\pi_1(\overline{Y})^{\ab, (p')}\rightarrow \pi_1(\overline{X})^{\ab,
	(p')}$ is surjective with finite kernel, the same argument as above
	shows that for almost all $n$ prime to $p$,
	$\Pic(\overline{X})[n]\xrightarrow{\cong}\Pic(\overline{Y})[n]$.
	Since $\Pic^0(\overline{X})$ is divisible, we have a short exact
	sequence
	\[
		0\rightarrow
		\Pic^0(\overline{X})[n]\rightarrow
		\Pic(\overline{X})[n]\rightarrow
		\NS(\overline{X})[n]\rightarrow
		0
		\] 
	which shows that for almost all $n$ prime to $p$,
	$\Pic^0(\overline{X})[n]\xrightarrow{\cong}\Pic^{0}(\overline{Y})[n]$.
	Since $\Pic^0(\overline{X})$ and $\Pic^0(\overline{Y})$ are sets of
	$k$-points of abelian varietes, and since pull-back along $f$ induces
	a morphism of the underlying abelian varieties, it follows that
	$\Pic^0(\overline{X})\rightarrow \Pic^0(\overline{Y})$ is surjective
	with finite kernel.  From the morphism of short exact sequences
	\begin{equation*}
		\xymatrix{
			0\ar[r]&K_X\cap
			\Pic^0(\overline{X})\ar[r]
			\ar[d]&\Pic^0(\overline{X})\ar[r]\ar[d]&
			\Pic^0(X)\ar[r]\ar[d]&0\\
			0\ar[r]&K_Y\cap
			\Pic^0(\overline{Y})\ar[r]
			&\Pic^0(\overline{Y})\ar[r]&
			\Pic^0(Y)\ar[r]&0
		}
	\end{equation*} 
	we see that $\Pic^0(X)\rightarrow \Pic^0(Y)$ is surjective.  To prove
	that its kernel is finite, it suffices to show that $\coker(K_X\cap
	\Pic^0(\overline{X})\rightarrow K_Y\cap \Pic^0(\overline{Y}))$ is
	finite.  For this, let $\ell$ be a prime different from $p$, and write
	$\overline{X}\setminus X =: D_X = \bigcup_{i=1}^{r(X)} D_{X,i}$ with
	$D_{X,i}$ smooth divisors.  Let $M_X:=\bigoplus D_{X,i}\mathbb{Z}$ be the
	free $\Z$-module of rank $r(X)$, and similarly for $Y$.  We have the
	associated Gysin exact sequence in \'etale cohomology \cite[Ch.~VI,
	Cor.~5.3]{Milne/EtaleCohomologyBook}:
	\begin{equation*}
		\xymatrix@C=.5cm{
			0\ar[r]&H^1(\overline{X},\Z_\ell(1))\ar[d]\ar[r]&H^1(X,\Z_\ell(1))\ar[d]\ar[r]&\ar[d]M_X\otimes
			\Z_\ell\ar[r]^--{c^{\ell,X}_1}&
			H^2(\overline{X},
			\Z_\ell(1))\ar[d]\\
			0\ar[r]&H^1(\overline{Y},\Z_\ell(1))\ar[r]&H^1(Y,\Z_\ell(1))\ar[r]&M_Y\otimes
			\Z_\ell\ar[r]^--{c^{\ell,
			Y}_1}&
			H^2(\overline{Y},
		\Z_\ell(1)) }
	\end{equation*} 
	As we have seen, for almost all primes $\ell$ the two left vertical
	arrows are isomorphisms, so
	$\ker(c_1^{\ell,X})\xrightarrow{\cong}\ker(c_1^{\ell,Y})$ for almost
	all $\ell$.  But by construction, we have a commutative diagram
	\begin{equation*}
		\xymatrix{
			\ker(c_1^{\ell,X})\ar@{->>}[r]\ar[d]&(K_X\cap
			\Pic^0(\overline{X}))\otimes
			\Z_\ell\ar[d]\\
			\ker(c_1^{\ell,Y})\ar@{->>}[r]&(K_Y\cap
			\Pic^0(\overline{Y}))\otimes
			\Z_\ell\\
		}
	\end{equation*} 
	This shows that the finitely generated group $\coker(K_X\cap
	\Pic^0(\overline{X})\rightarrow K_Y\cap \Pic^{0}(\overline{Y}))$ is
	trivial after tensoring with $\Z_\ell$ for almost all $\ell$ and hence
	finite. This finishes the proof that condition \ref{prop1} from  Lemma \ref{lemma:groupTheory},
	holds for diagram \eqref{diag:mainpic0}.

	To show that Lemma \ref{lemma:groupTheory}, \ref{prop2} holds for
	diagram \eqref{diag:mainpic0} we note that we have seen below
	\eqref{diag:pic0} that the kernel of
	$\Pic(X)\rightarrow \Pic(Y)$ contains no prime-to-$p$-torsion. 	

	Finally, lets check that Lemma \ref{lemma:groupTheory}, \ref{prop3} holds. Since $\Pic^0(X)$ is
	divisible, for every $n$ prime to $p$ we get a commutative diagram
	\begin{equation*}
		\xymatrix{ 
			\Pic(X)[n]\ar@{->>}[r]\ar[d]^{\cong}&\NS(X)[n]\ar[d]\\
			\Pic(Y)[n]\ar@{->>}[r]&\NS(Y)[n]
			 }
		 \end{equation*}
		 with surjective horizontal arrows. This shows that
		 $\NS(X)[p']\rightarrow \NS(X)[p']$ is surjective, so
		 Lemma \ref{lemma:groupTheory}, \ref{prop3} holds for diagram
		 \eqref{diag:mainpic0}.
\end{proof}
With Proposition \ref{prop:projlim}, the proof of Theorem
\ref{thm:rank1} becomes simple:
\begin{proof}[Proof of Theorem \ref{thm:rank1}]
	By Proposition \ref{prop:SESpicstrat}, $f$ induces a morphism of short
	exact sequences
	\begin{equation*}
		\xymatrix{
			0\ar[r]&\mathbb{I}(X)\ar[r]\ar[d]&\Pic^{\Strat}(X)\ar[r]\ar[d]&\plim
			\Pic(X)\ar[r]\ar[d]&
			0\\
			0\ar[r]&\mathbb{I}(Y)\ar[r]&\Pic^{\Strat}(Y)\ar[r]&\plim
			\Pic(Y)\ar[r]&
			0
		}
	\end{equation*}
	By assumption there exist good compactifications $\overline{X}$ and
	$\overline{Y}$,
	such that $f$ extends to
	$\bar{f}:\overline{Y}\rightarrow \overline{X}$.
	Then the right vertical morphism is an isomorphism by Proposition
	\ref{prop:projlim}.  We show that Proposition \ref{prop:descriptionofI} implies
	that the left vertical arrow is an isomorphism under our assumptions:   If
	$C:=\coker(H^0(X,\mathcal{O}_X^\times)\rightarrow
	H^0(Y,\mathcal{O}_Y^\times))$, then by assumption $C$ is a finitely
	generated abelian $p$-group,
	and hence a finite $p$-group.
	It follows that the diagonal map
	\[C\rightarrow
		\varprojlim_n
		C/p^n		\]
	is an isomorphism. By the exactness of $p$-adic completion of finitely
	generated abelian groups, we have a
	canonical isomorphism
	\[ \varprojlim_n C/p^n \cong\coker(\varprojlim_n
			H^0(X,\mathcal{O}_X^{\times})/p^n\rightarrow \varprojlim_n
			H^0(Y,\mathcal{O}_Y^{\times})/p^n),\]
so looking at the long exact sequence attached to the
	morphism of short exact sequences 
	\begin{equation*}
		\xymatrix{
			0\ar[r]&H^0(X,\mathcal{O}_X^\times)/k^\times\ar[d]
			\ar[r]&\varprojlim_n
			H^0(X,\mathcal{O}_X^{\times})/p^n\ar[d]_{\varprojlim
			f^*}\ar[r]&\mathbb{I}(X)\ar[r]\ar[d]&
			0\\
			0\ar[r]&H^0(Y,\mathcal{O}_Y^\times)/k^\times
			\ar[r]&\varprojlim_n
			H^0(Y,\mathcal{O}_Y^{\times})/p^n\ar[r]&\mathbb{I}(Y)\ar[r]&
			0,\\
		}
	\end{equation*}
	we see that
	$\mathbb{I}(X)\rightarrow \mathbb{I}(Y)$ is 
	surjective. But by Remark \ref{rem:globalunits}, \ref{rem:injectivity}
	the two left vertical arrows are injective, so $\mathbb{I}(X)\cong
	\mathbb{I}(Y)$ as claimed.
\end{proof}

To close this section, we show that in the case that $X=\Spec k$, the
assumption that $Y$ admits a good compactification is not necessary:

\begin{Proposition}\label{prop:rank1absolute}
	Let $Y$ be a smooth, connected, separated,
	finite type $k$-scheme, such that
	$\pi_1^{\ab,(p')}(Y)=0$.  Then
	$\Pic^{\Strat}(Y)=0$.
\end{Proposition}

In the proof of Proposition \ref{prop:rank1absolute} we will use Nagata's
theorem on compactifications (see \cite{Luetkebohmert/Compactification}) to
find a normal projective compactification $\overline{Y}$ of $Y$.  In a
preliminary version of this article we used de Jong's theorem on alterations
(\cite{deJong/Alterations}) to replace $\overline{Y}$ with a nice simplicial
scheme, and then we studied the attached groups of simplicial line bundles.
These arguments were long and technical.

Brian Conrad suggested working directly with the normal projective
compactifications instead of using simplicial techniques: He pointed out a
theorem by Lang stating that if $\overline{Y}$ is projective and normal, then
the group $\Cl^{\alg}(\overline{Y})$ of classes of those Weil divisors modulo linear
equivalence which are algebraically equivalent to $0$, is the group of
$k$-points of an abelian variety; for a modern treatment see
\cite{Gabber/Lang}.

We are grateful to Brian Conrad for bringing this result to our attention.

\begin{proof}[Proof of Proposition \ref{prop:rank1absolute}]
	By Proposition \ref{prop:globalunits}, we know that $\mathbb{I}(Y)=0$.
	Hence, according to Proposition \ref{prop:SESpicstrat}, we only have
	to show that $\plim \Pic(Y)=0$.
	As in \eqref{diag:pic0}, Kummer theory shows that $\Pic(Y)[n]=0$,
	whenever $n$ is prime to $p$. Thus,  Lemma
	\ref{lemma:R1} shows that it suffices to prove that $\Pic(Y)$ is a finitely
	generated group.

	To this end, let $\overline{Y}$ be a normal projective
	compactification of $Y$, and $\Cl(\overline{Y})$ the group of Weil
	divisors on $\overline{Y}$ modulo linear equivalence. We then have a
	surjection $\Cl(\overline{Y})\rightarrow \Pic(Y)$, and a short exact
	sequence
	\begin{equation*}
		\xymatrix{ 
			0\ar[r]&\Cl^{\alg}(\overline{Y})\ar[r]&\Cl(\overline{Y})\ar[r]&\NS(\overline{Y})\ar[r]&0
			 }
		 \end{equation*}
		 with $\NS(\overline{Y})$ finitely generated by
		 \cite[Thm.~3]{Kahn/GroupeDesClasses}. 
		 By the aforementionend
		 result of Lang, $\Cl^{\alg}(\overline{Y})$ is the set of
		 $k$-points of an abelian variety. More precisely, given a point $y\in
		 \overline{Y}(k)$, there is an abelian variety $A$ and a rational
		 map $\alpha_y: \overline{Y}\dashedrightarrow A$ defined
		 around $y$, such that
		 $\alpha_y(y)=0$, and such that every rational map
		 $\beta:Y\dashedrightarrow B$ defined around $y$, with $B$ an abelian variety and
		 $\beta(y)=0$, factors through $\alpha_y$. The abelian variety
		 $A$ is the Albanese
		 variety for rational maps, and its dual has the property that
		 $A^\vee(k)=\Cl^{\alg}(\overline{Y})$. For a modern treatment
		 see \cite{Gabber/Lang}.

		 Now to finish, note that since the kernel of the surjection
		 $\Cl(\overline{Y})\twoheadrightarrow \Pic Y$ is finitely generated,
		 it follows that $\Cl(\overline{Y})[n]=0$ for all but finitely
		 many $n$. This implies that $\Cl^{\alg}(\overline{Y})[n]=0$
		 for all but finitely many $n$, which shows that
		 $\Cl^{\alg}(\overline{Y})=0$, since the underlying abelian
		 variety must have dimension $0$. It follows that
		 $\Cl(\overline{Y})=\NS(\overline{Y})$ is finitely generated
		 and thus that $\Pic(Y)$ is finitely generated.
	 \end{proof}
\section{Special case II: Extensions of stratified bundles of rank $1$}\label{sec:extensions}
We continue to write $k$ for an algebraically closed field  of
characteristic $p>0$.

\begin{Lemma}\label{lemma:extensions}
	Let $X$ be a smooth, connected, separated $k$-scheme of finite type
	which admits a good compactification,
	and
	$\bar{x}$ a geometric point. If
	$\pi^{\tame}_1(X,\bar{x})$ does not have a quotient isomorphic to
	$\Z/p\Z$, then
	$\Ext^1_{\Strat^{\rs}(X)}(\mathcal{O}_X,\mathcal{O}_X)=0$.
\end{Lemma}
\begin{proof}
	Let $\overline{X}$ be a good compactification of $X$ and $E$ a regular
	singular stratified bundle on $X$, which is an extension of
	$\mathcal{O}_X$ by $\mathcal{O}_X$ in $\Strat^{\rs}(X)$. The aim is to
	show that $E$ extends to a stratified bundle $\overline{E}\in
	\Strat(\overline{X})$; the assumption on $\pi^{\tame}_1(X,\bar{x})$
	implies that $\pi_1(\overline{X},\bar{x})$ does not have a quotient
	isomorphic to $\Z/p\Z$, so  we can
	then apply \cite[Prop.~2.4]{EsnaultMehta/Gieseker} to show that
	$\overline{E}$ is trivial. 
	
	To show that $E$ extends to a stratified bundle $\overline{E}$ on
	$\overline{X}$, by Proposition \ref{prop:nonilpotent} it suffices to show that the
	exponents of $E$ along every component of the boundary divisor
	$\overline{X}\setminus X$ are $0\mod \Z$. Let $x_0\in
	\overline{X}$
	be a closed point, lying on precisely one
	component of $\overline{X}\setminus X$. Write
	$K_{x_0}:=\Frac(\widehat{\mathcal{O}_{\overline{X},x_0}})$. Then, after
	choosing coordinates $x_1,\ldots, x_n$, $K_{x_0}$ is isomorphic to the fraction
	field of the ring of formal
	power series $k\llbracket x_1,\ldots, x_n\rrbracket$. Write
	$\widehat{E}:=E\otimes
	K_{x_0}$. The stratification on $E$ gives $\widehat{E}$ the structure
	of a finite dimensional $K_{x_0}$-vector space with
	an action of the ring of differential operators
	$k[\partial_{x_i}^{(m)}|i=1,\ldots, n, m\geq 0]$, where the usual
	composition rules hold. Such an object is called an iterative
	differential module in \cite{Matzat}.
	The category
	$\Strat(K_{x_0})$ of such objects is
	``almost'' a neutral Tannakian category, but there might not exist a
	$k$-valued fiber functor (for more details see
	\cite[Ch.~1]{Kindler/thesis}). Fortunately, the sub-tensor category
	$\left<\widehat{E}\right>_{\otimes}\subset \Strat(K_{x_0})$ spanned by
	$\widehat{E}$ admits a fiber functor $\omega$ by
	\cite[Cor.~6.20]{Deligne/Festschrift}, as $k$ is algebraically closed.
	Composition of $\omega$ with the restriction functor
	$\left<E\right>_{\otimes}\rightarrow
	\left<\widehat{E}\right>_{\otimes}$ is a fiber functor for
	$\left<E\right>_{\otimes}$, hence we get a morphism
	$G(\left<\widehat{E}\right>_{\otimes})\rightarrow
	G(\left<E\right>_{\otimes})$  of the associated affine $k$-group
	schemes,
	and this morphism is a closed immersion by
	\cite[Prop.~2.21]{DeligneMilne}. 
	
	Since $E$ is an extension of
	$\mathcal{O}_X$ by $\mathcal{O}_X$, its monodromy group
	$G(\left<E\right>_{\otimes})$ is a closed subgroup scheme of
	$\G_{a,k}$. 
	But $E$ is assumed to be regular singular,
	so \cite[Thm.~3.3]{Gieseker/FlatBundles} implies that $\widehat{E}$ is a
	direct sum of rank $1$ objects of $\Strat(K_{x_0})$	and thus  $G(\left<\widehat{E}\right>_{\otimes})$ is a closed
	subgroup scheme of $\G_{m,k}^2$. 	We finally conclude that $\widehat{E}\cong K_{x_0}^2$ as an
	object of $\Strat(K_{x_0})$, since $\G_{a,k}$ does not have nontrivial
	diagonizable subgroups, and then
	\cite[Thm.~3.3]{Gieseker/FlatBundles} implies that the exponents of $E$ along the
	component on which $x_0$ lies are $0 \mod \Z$. We repeat the same
	argument for every component of the boundary $\overline{X}\setminus
	X$, and hence complete the proof.
\end{proof}
We can now prove one of the main results of this article:
\begin{Theorem}\label{thm:abelianQuotient}Let $X$ be a smooth, connected,
	separated $k$-scheme of finite type which admits a good
	compactification, and $\bar{x}$ a geometric point. If
	$\pi_1^{\tame}(X,\bar{x})^{\ab,(p')}=0$, and if
	$\pi_1^{\tame}(X,\bar{x})$ does not have a quotient isomorphic to
	$\Z/p\Z$, then every regular singular stratified bundle on $X$ which
	has abelian monodromy is trivial.
\end{Theorem}
\begin{proof}
	Let $E$ be a regular singular stratified bundle on $X$, and
	$\omega:\left<E\right>_{\otimes}\rightarrow \Vectf_k$ a fiber functor.
	Assume that the associated $k$-group scheme
	$G(\left<E\right>_{\otimes},\omega)$ is abelian. 
	By \cite[9.4]{Waterhouse} every irreducible object of
	$\left<E\right>_{\otimes}$ has rank $1$. By Proposition
	\ref{prop:rank1absolute} every rank $1$ object of $\Strat(X)$ is
	trivial, and by Lemma \ref{lemma:extensions} there are no nontrivial
	extensions between trivial objects of rank $1$. It follows that every
	object of $\left<E\right>_{\otimes}$ is trivial.\end{proof}

\section{Special case III: Affine spaces}\label{sec:affinespaces}
We continue to denote by $k$ an algebraically closed field of characteristic
$p>0$.  It is well known  that $\pi_1^{\tame}(\A^n_k)=0$ for all $n \geq 0$.
\begin{Theorem}\label{thm:affinespaces}
	Every regular singular stratified bundle on $\A^n_k$ is trivial.
\end{Theorem} 
\begin{Remark}\leavevmode
\begin{itemize} 
	\item A slightly different approach to Theorem \ref{thm:affinespaces}
		was  sketched in \cite[4.4]{Esnault/ECM} 
	\item Note that Theorem \ref{thm:affinespaces} follows directly
		from Proposition \ref{prop:rank1absolute} and
		\cite[Thm.~5.3]{Gieseker/FlatBundles}, which states that every regular
		singular stratified bundle on $\P^n_k\setminus D$ is a direct
		sum of stratified line bundles, if $D$ is a strict normal
		crossings divisor.  Unfortunately
		[loc.~cit.] is imprecisely stated
		(it is false for $n=1$) and its proof is very complicated. 

		Below we give a simple argument to prove Theorem
		\ref{thm:affinespaces},
		which is certainly implicitly
		contained in \cite{Gieseker/FlatBundles}.
\end{itemize}
\end{Remark}

\begin{proof}[Proof of Theorem \ref{thm:affinespaces}]
	The case of $\A^1_k$ follows from 
	\cite[Prop.~4.2]{Gieseker/FlatBundles} and Proposition
	\ref{prop:rank1absolute}. 

		We proceed by induction; let $n>1$.  Then the $n$-fold product
		$\P^1_k\times_k\ldots\times_k \P^1_k$ is a good
		compactification of $\A^n_k$, and if $E$ is a stratified
		bundle on $\A^n_k$, then $E$ is regular singular if and only
		if it is $(\A^n_k,(\P^1_k)^n)$-regular singular by Proposition
		\ref{prop:rsfacts}. 

		We
		compute
		the exponents of $E$ along the divisor
		$\left(\P^1_{k}\right)^{n-1}\times_k \{\infty\}\subset
		(\P^1_k)^n$. To this end let $\overline{E}$ be a free
		$\mathcal{O}_{\A^{n-1}_k\times (\P^1_{k}\setminus\{0\})}$-module
		with $\mathscr{D}_{\A^{n-1}_{k}\times (\P^1_{k}\setminus
		\{0\})/k}(\log \A^{n-1}_{k}\times\{\infty\})$-action extending
		$E|_{\A^{n-1}_k\times \G_m}$.
		Choose coordinates $x_1,\ldots, x_n$ such that
		$\A^{n-1}_k\times_k(\P^1_k\setminus \{0\})=\Spec
		k[x_1,\ldots, x_{n-1}, x^{-1}_n]$.  By Proposition
		\ref{prop:exponents} there exists a basis $e_1,\ldots, e_r$ of
		the free module $\overline{E}$, such that
			\[\delta_{x_n^{-1}}^{(m)}(e_i)=\binom{\alpha_i}{m}e_i+x_n^{-1}\overline{E}\]
		with $\alpha_i\in \Z_p$ an exponent of $\overline{E}$ along
		$\A^{n-1}_k\times \{\infty\}$.  But the same equation also
		holds modulo a prime ideal $(x_1-a_1,\ldots, x_{n-1}-a_{n-1})$,
		$a_1,\ldots, a_{n-1}\in k$, so $\alpha_i\mod \Z$ is an
		exponent of the fiber $E|_{(a_1,\ldots, a_{n-1})\times \A^1_k}$ along the
		divisor $(a_1,\ldots,a_{n-1},\infty)\subset (a_1,\ldots,
		a_{n-1})\times \P^1_k$. 

		The case $n=1$ now shows that $\alpha_i\equiv 0\mod \Z$, and
		hence the  exponents of $E$  along
		$\A^{n-1}_k\times_k \{\infty\}$ are $0\mod \Z$. By 
		Proposition \ref{prop:nonilpotent} this means that $E$ extends to an
		actual stratified bundle on $\A^{n-1}_k\times_k \P^1$.

		But now we are done: The above argument shows that $E$ extends
		to a stratified bundle on $(\P^1_k)^n$: First to $(\P^1_k)^n$
		minus a closed subset of codimension $\geq 2$, and then by
		\cite[Lemma 2.5]{Kindler/FiniteBundles} to $(\P^1_k)^n$.  But
		there are only trivial stratified bundles on $(\P^1_k)^n$, as
		it it is birational to $\P^n_k$, so $E$ is trivial (\cite[Thm.~2.2]{Gieseker/FlatBundles}).
\end{proof}
\section{Special case IV: Universal homeomorphisms}\label{sec:univHomeo}
We continue to denote by $k$ an algebraically closed
field of characteristic $p>0$.  Recall that by
\cite[18.12.11]{EGA4}, a finite type morphism
$f:Y\rightarrow X$ of finite type $k$-schemes is a
universal homeomorphism if and only if it is finite,
purely inseparable (i.e.~universally injective) and
surjective. It is proven in \cite[IX.4.10]{SGA1} that
$f$ induces an isomorphism
$\pi^{\et}_1(Y)\xrightarrow{\cong}\pi^{\et}_1(X)$ (with
appropriate choices of base points); it follows from
\cite{Vidal/TopologicalInvariance} that the same is
true for $\pi_1^{\tame}$.  In this section we prove
that pull-back along $f$ is an equivalence
$\Strat^{\rs}(X)\rightarrow \Strat^{\rs}(Y)$.

For a $k$-scheme $X$, we write $X^{(n)}$ for the base change of $X$ along the
$n$-th power of the absolute Frobenius of $k$, and by
$F_{X/k}^{(n)}:X\rightarrow X^{(n)}$ the associated $k$-linear relative
Frobenius.  It follows from Theorem \ref{thm:cartier} that pull-back along
$F^{(n)}_{X/k}$ induces an equivalence of categories
$\Strat(X^{(n)})\rightarrow \Strat(X)$.  This remains true in the regular
singular case: We first work with respect to one fixed good partial
compactification $(X,\overline{X})$.

\begin{Proposition}[``Frobenius descent'']\label{prop:frobeniusPullbackEquivalencePartialCompactification}
	If $(X,\overline{X})$ is a good partial compactification, then the
	pair $(X^{(n)},\overline{X}^{(n)})$ also is a good partial
	compactification, and $F_{X/k}^{(n)}$ induces an equivalence \[
		(F_{X/k}^{(n)})^*: \Strat^{\rs}(
		(X^{(n)},\overline{X}^{(n)}))\rightarrow \Strat^{\rs}(
		(X,\overline{X})).\]
\end{Proposition}
\begin{proof}
	We may assume that $n=1$, and write $F_{X/k}=F^{(1)}_{X/k}$.

	Write $D:=\overline{X}\setminus X$ and
	$D^{(1)}:=\overline{X}^{(1)}\setminus X^{(1)}$.  Since the functor
	$F_{X/k}^*:\Strat(X^{(1)})\rightarrow \Strat(X)$ is an equivalence, it
	suffices to show that the essential image of $\Strat^{\rs}(
	(X^{(1)},\overline{X}^{(1)}))$ in $\Strat(X)$ is $\Strat^{\rs}(
	(X,\overline{X}))$, i.e.~it suffices to show that a stratified bundle
	$E$ on $X^{(1)}$ is $(X^{(1)},\overline{X}^{(1)})$-regular singular if
	$F_{X/k}^*E$ is $(X,\overline{X})$-regular singular.

	Let $j:X\hookrightarrow \overline{X}$ denote the inclusion.  Assume
	that $F_{X/k}^*E$ is $(X,\overline{X})$-regular singular.  Let $E'$ be
	any torsion free coherent extension of $E$ to $\overline{X}^{(1)}$ and
	$\overline{E}$ the $\mathscr{D}_{\overline{X}^{(1)}/k}(\log
	D^{(1)})$-module generated by ${E'}$ in the
	$\mathscr{D}_{\overline{X}^{(1)}/k}(\log \mathscr{D}^{(1)})$-module
	$j^{(1)}_*E$.  We need to show that $\overline{E}$ is
	$\mathcal{O}_{\overline{X}^{(1)}}$-coherent.  Since
	$F_{\overline{X}/k}$ is faithfully flat, it suffices to show that
	$F_{\overline{X}/k}^*\overline{E}$ is
	$\mathcal{O}_{\overline{X}}$-coherent, \cite[Prop.~VIII.1.10]{SGA1}.
	Define $G$ to be the $\mathscr{D}_{\overline{X}/k}(\log D)$-module
	generated by $F_{\overline{X}/k}^*E'$ in
	$j_*F_{X/k}^*E=F^*_{\overline{X}/k}j^{(1)}_* E$.  Then $G$ is
	$\mathcal{O}_{\overline{X}}$-coherent by assumption, so the proof is
	complete if we can show that $G=F_{\overline{X}/k}^*\overline{E}$.

	The $\mathcal{O}_{\overline{X}^{(1)}}$-module $j^{(1)}_*E$ naturally
	carries a $\mathscr{D}_{\overline{X}^{(1)}/k}(\log D^{(1)})$-action,
	and similarly for $j_*F_{X/k}^*E$.  Then we can describe
	$F_{\overline{X}/k}^{*}\overline{E}$ as the image of the (pulled-back)
	evaluation morphism
	\[F^*_{\overline{X}/k}\left(\mathscr{D}_{\overline{X}^{(1)}/k}(\log
		D^{(1)})\otimes_{\mathcal{O}_{\overline{X}^{(1)}}}
		E'\right)\rightarrow
		F_{\overline{X}/k}^{*}j^{(1)}_*
		E,\]
	and $G$ as the image of the evaluation morphism
	\[
		\mathscr{D}_{\overline{X}/k}(\log
		D)\otimes_{\mathcal{O}_{\overline{X}}}
		F^*_{\overline{X}/k}E'\rightarrow
		j_*
		F_{X/k}^*E=F_{\overline{X}/k}^*j^{(1)}_*E.\]
	These morphisms fit in a commutative diagram (writing
	$F=F_{\overline{X}/k}$ for legibility):
	\begin{equation*}
		\xymatrix{
			F^*\left(\mathscr{D}_{\overline{X}^{(1)}/k}(\log
			D^{(1)})\otimes_{\mathcal{O}_{\overline{X}^{(1)}}}
			E'\right)\ar[r]&
			F^{*}j^{(1)}_*
			E\ar@{=}[ddd]\\
			F^*\mathscr{D}_{\overline{X}^{(1)}/k}(\log
			{D}^{(1)})\otimes_{F^{-1}\mathcal{O}_{\overline{X}^{(1)}}}F^{-1}
			E'\ar@{=}[u]\\
			\ar[u]^{\gamma\otimes
			\id}\mathscr{D}_{\overline{X}/k}(\log
			D^{(1)})\otimes_{F^{-1}\mathcal{O}_{\overline{X}^{(1)}}}F^{-1}E'\ar@{=}[d]\\
			\mathscr{D}_{\overline{X}/k}(\log
			D)\otimes_{\mathcal{O}_{\overline{X}}}
			F^*E'\ar[r]&
			j_*
			F^*E,
		}
	\end{equation*}
	where \[\gamma:\mathscr{D}_{\overline{X}/k}(\log D)\rightarrow
	F^*_{\overline{X}/k}\mathscr{D}_{\overline{X}^{(1)}/k}(\log
	D^{(1)})=\mathcal{O}_{\overline{X}}\otimes_{F^{-1}\mathcal{O}_{\overline{X}^{(1)}}}F^{-1}\mathscr{D}_{\overline{X}^{(1)}/k}(\log
	D^{(1)})\]
	is the canonical morphism coming via restriction from the morphism
	$\mathscr{D}_{\overline{X}/k}\rightarrow
	F^*_{\overline{X}/k}\mathscr{D}_{\overline{X}^{(1)}/k}$.  It follows
	that $G=F_{\overline{X}/k}^*\overline{E}$, if $\gamma$ is surjective. 

	This is a local question, so we may assume that $\overline{X}=\Spec A$
	and $X=\Spec A[t_1^{-1}]$, with $t_1,\ldots, t_r$ local coordinates on
	$A$.  Then $\overline{X}^{(1)}=\Spec A\otimes_{F_k} k$, and
	$t_1\otimes 1,\ldots, t_r\otimes 1$ is a system of local coordinates
	for $A\otimes_{F_k} k$.  The relative frobenius $F_{X/k}$ then maps
	$t_i\otimes 1\mapsto t_i^p$, for $i=1,\ldots, r$.

	With the notation
	$\delta_{t_i}^{(p^m)}$ from Definition \ref{defn:diffops},
	\ref{item:diffops}, and recalling that $\delta_{t_i}^{(p^m)}$ ``behaves
	like'' $\frac{t_i^{p^m}}{p^m!}\partial^{p^m}/\partial
		t_i^{p^m}$,
		we claim that $\gamma(\delta^{(1)}_{t_i})=0$ and 
	\begin{equation}\label{eqn:frobOps}
		\gamma(\delta_{t_i}^{(p^m)}) = \delta^{(p^{m-1})}_{t_i\otimes
		1} \text{ for } m>1,
	\end{equation}
	which shows that the image of $\gamma$ contains all
	the generators of the
	left-$\mathcal{O}_{\overline{X}}$-algebra
	$F_{X/k}^*\mathscr{D}_{\overline{X}^{(1)}/k}(\log
	D^{(1)})$, and thus that $\gamma$ is surjective.  As
	for the claim, it suffices to observe that for $s\geq
	0$,
	\[\delta^{(p^m)}_{t_i}(
		(t_j^p)^{s})=
		\begin{cases}
			0
			&
			i\neq
			j\\
			\binom{sp}{p^m}&\text{otherwise},
		\end{cases}
		\]
	and finally that
	\[\binom{sp}{p^m}\equiv\binom{s}{p^{m-1}}\mod
			p.\]
\end{proof}
\begin{Remark}
	This proof does \emph{not} show that all
	$\mathcal{O}_{\overline{X}}$-coherent
	$\mathscr{D}_{\overline{X}/k}(\log D)$-module descend to
	$\mathcal{O}_{\overline{X}^{(1)}}$-coherent
	$\mathscr{D}_{\overline{X}^{(1)}/k}(\log D^{(1)})$-modules.
	In fact, such a statement is
	false, due to the failure of Cartier's theorem \cite[\S
	5.]{Katz/Connections} for logarithmic connections. 
	On the other hand,  in 
	\cite{Lorenzon}  a version of Cartier's theorem for
	$\log$-schemes is developed, which is applied in \cite[Ch.~4]{Montagnon} to
	construct a generalization of Frobenius descent to the logarithmic
	setting. For this however, the rings of coefficients have to
	be enlarged.
\end{Remark}
\begin{Corollary}\label{cor:frobMultipliesExponents}
	With the notations from Proposition
	\ref{prop:frobeniusPullbackEquivalencePartialCompactification},
	if $\overline{E}$ is a locally free
	$\mathcal{O}_{\overline{X}^{(1)}}$-coherent
	$\mathscr{D}_{\overline{X}^{(1)}/k}(\log D^{(1)})$-module with
	exponents $\alpha_1,\ldots, \alpha_n\in \Z_p$, then
	$F^*_{\overline{X}/k}\overline{E}$ is a
	$\mathscr{D}_{\overline{X}/k}(\log D)$-module with exponents
	$p\alpha_1,\ldots, p\alpha_n$.

	Consequently, if $E\in \Strat^{\rs}(
	(X^{(1)},\overline{X}^{(1)}))$, with exponents
	$\alpha_1,\ldots, \alpha_n\in \Z_p/\Z$, then $F_{X/k}^{*}E$
	has exponents $p\alpha_1,\ldots, p\alpha_n\in \Z_p/\Z$.
\end{Corollary}
\begin{proof}
	The claim follows directly from the formula \eqref{eqn:frobOps} and
	the fact that $\delta_{t_1}^{(1)}$ acts on
	$F_{\overline{X}/k}^*\overline{E}=\mathcal{O}_{\overline{X}}\otimes_{F^{-1}\mathcal{O}_{\overline{X}}^{(1)}}
	\overline{E}$ via $\delta^{(1)}_{t_1}(a\otimes
	e)=\delta^{(1)}_{t_1}(a)\otimes e$.
\end{proof}
\begin{Remark}
	Note that multiplication by $p$ is an automorphism of the group
	$\Z_p/\Z$.
\end{Remark}

\begin{Theorem}\label{thm:partiallyCompactifiedTopInvariance}
	Let $(X,\overline{X})$, $(Y,\overline{Y})$ be good partial
	compactifications and $\bar{f}:\overline{Y}\rightarrow \overline{X}$ a
	universal homeomorphism such that $\bar{f}(X)\subset \overline{Y}$.
	If we write $f:=\bar{f}|_X$ then $f$ induces an equivalence
	\[f^*:\Strat^{\rs}( (X,\overline{X}))\xrightarrow{\cong} \Strat^{\rs}(
		(Y,\overline{Y})).\]
\end{Theorem}
\begin{proof}
	The morphism $\bar{f}$ is finite of degree $p^n$ for some $n$,
	and thus there is a morphism $\bar{g}:\overline{X}\rightarrow
	\overline{Y}^{(n)}$, such that
	$\bar{g}\bar{f}=F_{\overline{Y}/k}^{(n)}$, and such that
	$\bar{f}^{(n)}\bar{g}=F_{\overline{X}/k}^{(n)}$.  Moreover,
	$\bar{g}(X)\subset Y^{(n)}$.  Write $g:=\bar{g}|_X$.  Then
	$(g{f})^*$ is an equivalence by Proposition
	\ref{prop:frobeniusPullbackEquivalencePartialCompactification},
	so $f^*$ is essentially surjective.  But $(f^{(n)}g)^*$ also
	is an equivalence, so $f^{(n),*}$ is full.  This shows that
	$f^*$ is full, and since $f$ is faithfully flat, it follows
	that $f$ is faithful as well.  This finishes the proof.
\end{proof}

\begin{Theorem}\label{thm:topologicalInvarianceGeneralCase}
	Let $f:Y\rightarrow X$ be a universal homeomorphism of smooth,
	separated, finite type $k$-schemes.  Then $f$ induces an
	equivalence
	\[f^*:\Strat^{\rs}(X)\xrightarrow{\cong}\Strat^{\rs}(Y).\]
\end{Theorem}
\begin{proof}
	Without loss of generality we may assume that $X, Y$ are connected.
	The same argument as in the proof of Theorem
	\ref{thm:partiallyCompactifiedTopInvariance} shows that the fact that
	the relative Frobenius induces an equivalence $\Strat(X^{(n)})\rightarrow
	\Strat(X)$, implies that the functor
	$f^*:\Strat(X)\rightarrow \Strat(Y)$ is an equivalence.

	If follows that we just need to check that $f^*E$ is
	regular singular for $E\in \Strat(X)$, if and only if $E$ is regular
	singular. 

	Assume that $E$ is regular singular and let $(Y,\overline{Y})$ be a
	good partial compactification with $\overline{Y}\setminus Y$ smooth
	with generic point $\eta$. Using that $k(X)\subset k(Y)$ is purely
	inseparable, one quickly checks that
	$R:=\mathcal{O}_{\overline{Y},\eta}\cap k(X)$ is a discrete valuation
	ring, that $\mathcal{O}_{Y,\eta}$ is the integral closure of $R$ in
	$k(Y)$, and hence that the residue field of $R$ has transcendence 
 degree $\dim X - 1$ over
	$k$. This means that $R$ is the local ring of a codimension $1$ point
	on some model of $k(X)$. Hence there exists a good partial compactification
	$(X,\overline{X})$, such that $f$ extends to a morphism
	$\bar{f}:\overline{Y}\rightarrow\overline{X}$ (after possibly removing
	a closed subset of codimension $\geq 2$ from $\overline{Y}$). This shows that $f^*E$
	is $(Y,\overline{Y})$-regular singular. We repeat this for every good
	partial compactification $(Y,\overline{Y})$ to conclude that $f^*E$ is
	regular singular.

	Conversely, assume that $f^*E$ is regular singular.
	To prove that $E$ is regular singular we need to show that for every good partial
	compactification $(X,\overline{X})$, such that
	$\overline{X}\setminus X$ is smooth with generic point $\xi$,
	there exists an open neighborhood $\overline{U}'$ of $\xi$, such that
	$E$ is $(\overline{U}'\cap X, \overline{U}')$-regular singular. 

	Let $\overline{U}=\Spec \overline{A}$ be an affine open neighborhood of
	$\xi$. We may assume that $U:=\overline{U}\cap X$ is also affine, say
	$U=\Spec A$.  Because $f$ is finite,
	$V:=f^{-1}(U)$ is affine, say $V=\Spec B$, and $f|_V$ is
	a universal homeomorphism.  Let $\overline{B}$ be the integral
	closure of $\overline{A}$ in $B$, and $\overline{V}:=\Spec
	\overline{B}$.  Then $\bar{g}:\overline{V}\rightarrow \overline{U}$ is a finite morphism.  By
	construction $\overline{V}$ is normal,
	and $\bar{g}$ is
	a universal homeomorphism, because $k(X)\subset k(Y)$ is purely
	inseparable.  We may shrink $\overline{U}$ around $\xi$
	to obtain an open neighborhood $\overline{U}'\subset
	\overline{U}$ of $\xi$, such that
	$\overline{V}':=\bar{g}^{-1}(\overline{U}')$ is
	smooth, and $\bar{g}':\overline{V}'\rightarrow \overline{U}'$ is
	a universal homeomorphism. Moreover, writing $V':=\overline{V}'\cap
	Y$, $U':=\overline{U}'\cap X$, we see that
	$\bar{g}'|_{V'}=f|_{V'}:V'\rightarrow U'$. Since $(f|_{V'})^*(E|_{U'})$ is
	$(V',\overline{V}')$-regular singular by assumption,  we can apply Theorem
	\ref{thm:partiallyCompactifiedTopInvariance}  to see that $E$ is
	$(U',\overline{U}')$-regular singular.
\end{proof}

\providecommand{\bysame}{\leavevmode\hbox to3em{\hrulefill}\thinspace}
\providecommand{\MR}{\relax\ifhmode\unskip\space\fi MR }
\providecommand{\MRhref}[2]{%
  \href{http://www.ams.org/mathscinet-getitem?mr=#1}{#2}
}
\providecommand{\href}[2]{#2}

\Addresses
\end{document}